\documentclass[10pt,reqno,bibliography=leveldown]{amsart}

\usepackage[hmarginratio=1:1,top=32mm,columnsep=20pt]{geometry} 
\usepackage{multirow}
\usepackage{diagbox}
\usepackage[dvipsnames]{xcolor}
\usepackage{xpatch} 
\usepackage{mathrsfs}  
\usepackage{xxcolor}
\usepackage{graphicx} 
\usepackage{booktabs}
\usepackage{mystuff,mymathenv} 
\usepackage{delimdelim} 
\usepackage{esint}
\graphicspath{{figs/}}
\usepackage[
    backend=bibtex,
    style=numeric,
    giveninits = true,
    maxbibnames=5,
    url=false, 
    doi=true,
    eprint=true
    ]{biblatex}
    
    \newtheorem*{maintheorem*}{Main Result}

    \newcommand{\ep}{\epsilon}
    \newcommand{\m}{\,\mbox{d}}
    
    \newcommand{\f}{\boldsymbol}

    \newcommand{\domain}{D}
    \newcommand{\potential}{\mathcal{V}}

    \newcommand{\mass}{{m_{\rho}}}
    \numberwithin{equation}{section}
    
    \crefformat{ineq}{(#2#1#3)}
    \crefmultiformat{ineq}{(#2#1#3)}%
    { and~(#2#1#3)}{, (#2#1#3)}{ and~(#2#1#3)}  
    
    \crefname{figure}{Figure}{Figures}%
    \title[The RIDK equation: DG approximation ... modelling for low-density regime]{The Regularised Inertial Dean--Kawasaki equation: Discontinuous Galerkin approximation and modelling for low-density regime}
    \author{Federico Cornalba}
    \address{Institute of Science and Technology Austria (ISTA), Am~Campus~1, 
    3400 Klosterneuburg, Austria}
    \email{federico.cornalba@ist.ac.at}
    \author{Tony Shardlow}
    \address{University of Bath, BA2 7AY, United Kingdom}
    \email{t.shardlow@bath.ac.uk}
    
    \date{\today}
\begin{document}
    \maketitle
    \begin{abstract}

The Regularised Inertial Dean--Kawasaki model (RIDK) -- introduced by the authors and J. Zimmer in earlier works -- is a nonlinear stochastic PDE capturing fluctuations around the mean-field limit for large-scale particle systems in both particle density and momentum density. 

We focus on the following two aspects. Firstly, we set up a Discontinuous Galerkin (DG) discretisation scheme for the RIDK model: we provide suitable definitions of numerical fluxes at the interface of the mesh elements which are consistent with the wave-type nature of the RIDK model and grant stability of the simulations, and we quantify the rate of convergence in mean square to the continuous RIDK model. 
Secondly,  we introduce modifications of the RIDK model in order to preserve positivity of the density (such a feature only holds in a ``high-probability sense'' for the original RIDK model). By means of numerical simulations, we show that the modifications lead to physically realistic and positive density profiles. In one case, subject to additional regularity constraints, we also prove positivity.  Finally, we present an application of our methodology to a system of diffusing and reacting particles.

Our Python code is available in open-source format.
\end{abstract}
\section{Introduction}
The Regularised Inertial Dean--Kawasaki model (RIDK; see \cite{cornalba2021well}) is a stochastic PDE describing fluctuations of large-scale particle systems, which, crucially, are of inertial type. 
Specifically, RIDK not only keeps track of the particle density, but also the particle momentum density. It was originally derived from inertial Langevin dynamics, which is an established and accurate microscopic representation for a wide range of phenomena such as, but non limited to, active matter \cite{cates2015motility}, nucleation for colloids/thermal advection \cite{lutsko2012dynamical}, thin-liquid films rupture \cite{duran2019instability}, density/agent-based models \cite{helfmann2021interacting,djurdjevac2022feedback}, bacterial dynamics \cite{thompson2011lattice}.

In order to give minimal context to RIDK, let a density $\rho_\epsilon$ and momentum density $\vec{j}_\epsilon$ be defined by 
\begin{align}\label{ParticleDensities}
  \rho_\ep(\f{x},t) \coloneq \frac 1N 
  \sum_{i=1}^{N}{w_\ep(\f{x}-\f{q}_i(t))}, \quad \f{j}_\ep(\f{x},t) \coloneq \frac 1N\sum_{i=1}^{N}{\f{p}_i(t)w_\ep(\f{x}-\f{q}_i(t))},
\end{align}
for a smoothed delta function $w_\epsilon$ and regularisation parameter $\epsilon > 0$,
associated with an $N$-particle system, with positions and velocities $\{\f{q}_i,\f{p}_i\}_{i=1}^{N}$ undergoing inertial Langevin dynamics for potential energy $\potential$, dissipation parameter $\gamma > 0$ and noise intensity $\sigma\in\mathbb{R}$. 
On the periodic domain $D=\mathbb{T}^d$, for $\mathbb{N} \ni d \geq 1$, RIDK is the system of stochastic PDEs for $(\rho,\vec j)\approx (\rho_\epsilon,\vec j_\epsilon)$ given by
\begin{subequations}
  \makeatletter
  \def\@currentlabel{RIDK}
  \makeatother
  \label{ridk}
  \renewcommand{\theequation}{RIDK.\arabic{equation}}
  \begin{align}
    \frac{\partial\rho}{\partial t}
    &=-\nabla \cdot \vec{j},\label{ridk-rho}\tag{RIDK-$\rho$}\\
    \frac{\partial \vec{j}}{\partial t}
    &= -\gamma\,\vec{j} - k_B T\,\nabla\rho -\nabla \potential \,\rho+ \sigma\,\frac{1}{\sqrt N}\,\sqrt{\rho}\,\vec\xi_\ep,\label{ridk-j}\tag{RIDK-$\f{j}$}
  \end{align}
\end{subequations}
where $k_{B}T = \sigma^2/2\gamma$ (the fluctuation-dissipation property) and the Gaussian noise $\vec \xi_\epsilon\in\reals^d $ has independent, mean zero, white-in-time and correlated-in-space components with spatial covariance kernel of \emph{von Mises}-type (see \eqref{Covariance}; this  kernel is the periodic analogous of a Gaussian kernel with variance $\ep^2$). Subject to technical constraints on the initial data, the system is well-posed (see \cite{cornalba2021well}).
Note that \eqref{ridk} may also include additional terms such as particles reacting or interacting weakly according to a pair potential. 

\subsection{Main results}\label{MainResultRepo} We further consolidate the analysis of \eqref{ridk} by addressing two important aspects, specifically:
\begin{itemize}
\item we provide a numerical approximation of \eqref{ridk} by means of the discontinuous Galerkin method in space (more on this in Subsection \ref{num_section}), and
\item we improve modelling aspects in the low $\rho$-density regime (more on this in Subsection \ref{modelling_section}). 
\end{itemize}

\subsubsection{A Discontinuous Galerkin (DG) framework}\label{num_section}
The density $\rho$ is governed by a conservation law, and we choose a DG approximation to ensure local conservation of $\rho$.  
We derive and prove convergence of a Raviart--Thomas mixed finite-element approximation in space. The Raviart--Thomas elements are an important class of discontinuous basis functions and are the minimal set  of elements that are mapped  by the divergence operator  onto the piecewise polynomials. Key to defining a DG method is the numerical flux, which defines the flow between individual elements. 
 The numerical flux (given in \eqref{flux} below) is found by solving a wave equation across element edges and  therefore depends on the wave-speed $\sqrt{k_BT}$ and jump quantities $\jump{\rho}$. We summarise the convergence result (for a full statement, see \cref{ErrorBoundStochastic}).
\begin{proposition}[Informal statement for Proposition \ref{ErrorBoundStochastic}]
  Let $u=(\rho,\f{j})$ (respectively, $u_h=(\rho_h,\f{j}_h)$) be the solution to \eqref{ridk} (respectively, its DG approximation on a mesh with mesh-width $h$, with $q$ being the order of the chosen finite element discretisation, see \eqref{dg_ridk}) up to some fixed time $T>0$. Set 
   \begin{align}\label{ConvergenceOrderSquared}
    \tilde{q} \coloneq
    \left\{
      \begin{array}{rl}
        1/2 & \mbox{if } q = 0, \\
        q & \mbox{if }q>0.
      \end{array}
      \right.
    \end{align}
  Assume the validity of the scaling
\begin{align}\label{ScalingNEpsilon}
N\ep^{\theta} = 1, \qquad \theta \geq 2\overline{s}+ d,\qquad\mbox{for some }\overline{s} > \max\{d/2+1; q+3\}.
\end{align}
Furthermore, assume that the noise in \eqref{dg_ridk} is obtained as a truncation of the full \eqref{ridk} noise on the first $J\propto \ep^{-1}\ln(h^{2\tilde{q}})$ Fourier modes, as detailed in \eqref{TruncationIndexSet}.
  Then we have 
  \begin{align*} 
    \sup_{0\le t\le T}
    & \mean{\norm{R_hu(t) - u_h(t) }^2_{L^2(\mathbb{T}^d)}}  \le C(\delta,T,d) \left\{1+ \mean{\|u_0\|^2_{H^{\overline{s}} \times [H^{\overline{s}}]^d}}\right\} e^{C_2\left(\mathcal{V},T\right)}\;h^{2\tilde{q}},
  \end{align*} 
  where $u_0$ is the initial datum, where $R_h$ is a suitable projection operator, and where $\delta$ is a suitable regularisation parameter (see Remark \ref{rem:1}).
   \end{proposition}
  The justification for \eqref{ScalingNEpsilon} is technical, and is deferred to Remark \eqref{JustifyScaling}.
  
  Our implementation \cite{gitrepo} relies on a semi-implicit Euler--Maruyama time-stepping method, and makes use of the \texttt{Python} package \texttt{Firedrake}~\cite{firedrake}. We use this implementation to illustrate the behaviour of RIDK; for a one-dimensional example, see \cref{fig:intro}, which shows four snapshots across the time interval $[0,10$] of the $x$-profile of $\rho$ with initial data 
  \begin{equation}
    j_0(x)=0,\qquad \rho_0(x)=[2\pi(1+\pi)]^{-1}(1+x),\quad x\in(0,2\pi],\label{init}
  \end{equation} 
  (note that $\rho_0$ has unitary mass) and parameters 
  \begin{equation}
    \gamma=0.25,\quad\sigma=0.25,\quad\epsilon=0.05,\quad N=10^3,\quad \potential(x)=0.5 \cos(x)^2.
  \label{param}\end{equation}
  \begin{figure}[h] 
    \centering
    \includegraphics{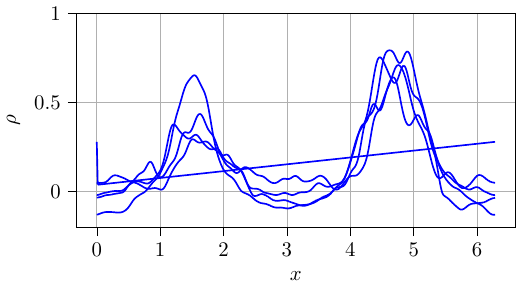}
    \caption{Five snapshots of the particle density $\rho(\cdot,t)$ for \eqref{ridk} at $t=0,2.5,5,7.5,10$ with parameters defined by \cref{init,param}.}
    \label{fig:intro}
    \end{figure}
  
Conspicuously, the particle density $\rho$ in \cref{fig:intro} becomes negative and this is not physical.  The non-negativity is not caused by the numerical approximation, but rather is a fundamental limitation of the RIDK model; in fact, RIDK is akin to a damped wave equation and therefore has no maximum principle or guarantee of positivity for the density (more on this in subsection \ref{RelatedLiterature}).

\subsubsection{Modelling in low-density regime}\label{modelling_section}
  We propose modelling modifications to \eqref{ridk} that lead to positive density profiles. 
  The most promising one (see Subsection \ref{sec:RIDK-switch}) is concerned with separating the time scale of density and momentum density, and speeding up the dynamics of the latter in the low density regime. This results in excessive momentum (which may lead to negativity) to quickly dissipate. Explicitly, we consider the modified system in $(\rho,\vec J)$
\begin{subequations}
\makeatletter
        \def\@currentlabel{Mod}
        \makeatother
        \label{Mod}
        \renewcommand{\theequation}{Mod.\arabic{equation}}
\begin{align}
\frac{\partial\rho}{\partial t}
&=-\nabla \cdot \vec J,\label{Mod2rho}\tag{Mod-$\rho$}\\
\varphi_\tau(\rho)\frac{\partial \vec J}{\partial t}
&= -\gamma\,\vec J - k_BT\,\nabla\rho -\nabla \potential\,\rho^+ + \sigma\,\frac{1}{\sqrt N}
\,\sqrt{\rho\,\varphi_\tau(\rho)}\,\vec\xi_\ep,\label{Mod2j}\tag{Mod-$\f{J}$}
\end{align}  
\end{subequations}
where $\rho^+ \coloneq \max\{\rho,0\}$ and  $\varphi_\tau$ is a function smoothly transitioning from value zero (on the $(-\infty,\tau/2)$ interval) to value one (on the $(\tau, \infty)$ interval). 
We discuss the extent to which \eqref{Mod} preserves positivity of the density. In particular, we provide a maximum principle-based argument that guarantees non-negativity (nevertheless, this setting is so far limited by the lack of a well-posedness theory), discuss associated numerical approximations, and show relevant simulations.
Several related open questions are discussed.

In a second approach, instead of speeding up the momentum dynamics, we add extra diffusion to the density evolution. This approach appears to be less successful and numerical simulations suggest strong dependence of the results on the size of the added diffusion. A brief discussion in given in Subsection \ref{sec:ridk-mix-coeff}.

Finally, we present a two-dimensional example with two populations of reacting/diffusing agents; specifically, we set up the corresponding RIDK dynamics, and compare its behaviour with that of the agents' system: this example follows the setting of the over-damped counterpart treated in \cite{helfmann2021interacting}.

\subsection{Related Literature}\label{RelatedLiterature}

\subsubsection{Inertial models}

Interest in analysis and simulation of DK-type equations has grown substantially during the last decade, with applications ranging from  active matter \cite{cates2015motility}, nucleation for colloids/thermal advection \cite{lutsko2012dynamical}, thin-liquid films rupture \cite{duran2019instability}, density/agent-based models \cite{helfmann2021interacting,djurdjevac2022feedback}, bacterial dynamics \cite{thompson2011lattice}.
As far as inertial Dean--Kawasaki models are concerned, analytical well-posedness of \eqref{ridk} (in the form of existence of high-probability mild solutions) has been addressed in the case of independent particles in dimension $d=1$ \cite{Cornalba2019a}, weakly interacting particles in dimension $d=1$ \cite{Cornalba2021al}, and in any dimension $d\geq 1$ with optimal scaling \cite{cornalba2021well}.  
In terms of numerical works, we cite the finite-element discretisations of the inertial models \cite{helfmann2021interacting, kim2017stochastic}
 for the general fluctuating hydrodynamics approximation for reaction-diffusion and agent-based systems, as well as the more specific work \cite{djurdjevac2022feedback} for the description of co-evolving opinion and social dynamics within agent-based systems.

\subsubsection{Trade-off: RIDK versus original Dean--Kawasaki model}

The \eqref{ridk} model is the inertial counterpart to the classical (over-damped) Dean--Kawasaki model~\cite{Dean,Kawasaki} in particle density only (corresponding to formally taking the limit $\gamma\rightarrow \infty$ in \eqref{ridk}), which reads
\begin{align}\label{dk}\tag{DK}
\frac{\partial \rho}{\partial t} = \frac{N}{2}\Delta \rho + \nabla \cdot \left(\sqrt{\rho}\,\f{\xi}\right),
\end{align}
where $\vec \xi$ is a space-time white noise. While they both describe fluctuations in systems of finitely many particles, \eqref{ridk} and \eqref{dk} are substantially different mathematical objects, and have different strengths and weaknesses. On one hand, the \eqref{ridk} model is more advantageous than \eqref{dk}, as it: 
\begin{enumerate}[i)]
\item allows for a richer description of the particle system, as it also includes the momentum density; 
\item has a more interpretable, less mathematically challenging noise (it is not in divergence form);
\item allows for smoother solutions, and;
\item features densities with only position $\f{x}$ and time $t$ as independent variables, thus retaining the same interpretability of the \eqref{dk} model\footnote{Closed mesoscopic representations of under-damped particle systems can be derived by including the velocity as an independent variable in the mesoscopic densities: this is precisely what is done for Vlasov--Fokker--Planck systems (see, e.g., \cite{duong2013generic}).}. 
\end{enumerate}
On the other hand, \eqref{ridk} loses out to \eqref{dk} when it comes to the regularity of the deterministic component (heat vs. wave-type drift): This is one of the main reasons for which the \eqref{dk} has been, so far, more extensively studied. Furthermore, unlike \eqref{dk}, the derivation of \eqref{ridk} relies on a close-to-equilibrium assumption for the underlying particle system: namely, such an assumption allows to compare the gradient term $- k_B T\,\nabla\rho $ in \eqref{ridk} and the microscopic term 
\begin{align}\label{NotClosableJ2}
[\f{j}_{\ep,2}(\f{x},t)]_\ell \coloneq - N^{-1}\sum_{i=1}^{N}{\sum_{k=1}^{d}{p_{\ell,i}(t)p_{k,i}(t)\frac{\partial}{\partial x_\ell} w_\ep(\f{x}-\f{q}_i(t))}}, \quad \ell = 1,\dots,d,
\end{align}
which arises from a time differentiation of -- and is otherwise not closable in terms of -- the densities in \eqref{ParticleDensities}; essentially, the identification of $- k_B T\,\nabla\rho $ and \eqref{NotClosableJ2} is achieved by assuming that the velocities $\{\f{p}_i\}_{i=1}^{N}$ are close to the equilibrium with variance  $k_BT$, see \cite{cornalba2021well} for an expanded discussion. Unsurprisingly, this comparison deteriorates for regimes of low particle density, leading to negativity of the density, as already shown in \cref{fig:intro}.

As already mentioned, \eqref{ridk} and \eqref{dk} are mathematically quite different: More specifically, recovering \eqref{dk} from \eqref{ridk} can currently be done only to a very limited extent: Specifically, we are not aware of any rigorous result quantifying the -- formal -- over-damped limit $\gamma\rightarrow \infty$. On the contrary, the two models are better understood on their own in terms of the limit of $N\rightarrow \infty$ (with \eqref{dk} recovering the heat equation) and the simultaneous limit of $N\rightarrow \infty,\ep\rightarrow 0$ as prescribed by \eqref{ScalingNEpsilon} (with \eqref{ridk} recovering the wave equation). Finally, \eqref{dk} is also understood when it comes to large deviations principle when associated with removing regularisations for the noise square-root singularity at the origin \cite{fehrman2019well}

\subsubsection{Over-damped Dean--Kawasaki model}

The seminal results \cite{konarovskyi2019dean,konarovskyi2020dean} show that for \eqref{dk} -- and natural variations associated with weakly interacting particle systems -- the only admissible martingale solution is precisely the empirical distribution of the underlying particle system. This boils down to the rigid interplay of (highly singular) noise and drift in the setting of a stochastic Wasserstein gradient flow. These results were preceded by -- and are the natural outcome of -- a bulk of works indicating the need for various regularisations in the equation's drift in order to obtain non-trivial solutions \cite{andres2010particle,von2009entropic,konarovskyi2017reversible,konarovskyi2019modified}.
DK models have also been linked to large-deviation principles in more general settings (for instance, zero-range processes, see \cite{dirr2016entropic}).

A second group of works is related to analytical regularisations of \eqref{dk} equations (coming from using either coloured, truncated, or smoothed noise). We mention 
fluctuation weak error estimates for a regularised DK SPDE started from general initial particle profiles, along with non-negativity of the solution, comparison principle, entropy estimates \cite{djurdjevac2022weak},
existence of kinetic solutions for regularised versions of \eqref{dk} and generation of a random dynamical system \cite{fehrman2019well}, 
uniqueness of invariant measures and mixing for the corresponding Markov process \cite{fehrman2022ergodicity},
derivation of large-deviation principles for fluctuations of the symmetric simple-exclusion process \cite{dirr2020conservative}, 
well-posedness for versions of \eqref{dk} with correlated noise \cite{fehrman2021well}, and derivation of underlying particle dynamics corresponding to \eqref{dk} with correlated noise \cite{ding2022new}. 
Additionally, rates of CLT convergence of stochastic gradient descent dynamics in overparametrised shallow neural networks to conservative stochastic PDEs close to \eqref{dk} have been recently derived in \cite{gess2022conservative}.

As for numerical discretisations of \eqref{dk}, we mention structure-preserving finite-element and finite-difference schemes for approximating the fluctuation density of $N$ particles to arbitrary precision in $N^{-1}$ (\cite{cornalba2021dean} for independent Brownian motions, \cite{cornalba2023density} for weakly interacting particles), 
convergence analysis of a finite element approximation to a weak formulation for a regularised \eqref{dk} equation \cite{bavnas2020numerical},
full reconstruction of dissipative operators in gradient flow equations \cite{li2019harnessing}, 
finite-volume schemes for stochastic gradient flow equations with hybrid space discretisation of the deterministic and the stochastic DK dynamics taking advantage of both central and upwind schemes, and positivity-preserving schemes based on Brownian bridge techniques \cite{russo2021finite}, and finite-volume schemes for stochastic gradient flow equations with applications to Landau–Lifshitz Navier–Stokes equations \cite{donev2010accuracy}.
\begin{remark}
Recently, we have seen several authors turn to Discontinuous Galerkin methods as a way to numerically simulate stochastic PDE models, including conservation laws (see, e.g., the review paper \cite{li2020discontinuous}).
\end{remark}
 
\subsection{Summary of Contents} 
We recall useful properties of \eqref{ridk} in section \ref{ExistenceTheory}, and we set up the DG method in section \ref{GalerkinFramework} for approximating RIDK spatially. We study a noise-free, linear problem associated to \eqref{ridk} in section \ref{ErrorAnalysisLinearPDE}, and quantify the convergence of our DG approximation in space for the full system in section \ref{ErrorAnalysisRIDK}. In section \ref{Positivity}, we propose modifications to \eqref{ridk} in order to address the aforementioned out-of-equilibrium and density-positivity issues: these modifications are discussed both analytically and numerically. Finally, 
section \ref{extended} provides a comparison between a reacting/diffusing system of two populations of agents (in inertial form) and the associated RIDK model.

The appendix is devoted to technical lemmas for the \eqref{ridk} model (section \ref{AppRIDKSelectedFeatures}), 
and computing relevant It\^o differentials (section \ref{RelevantItoDifferentials}).

\section{well-posedness for ridk model}\label{ExistenceTheory}

We recall relevant notions from \cite{cornalba2021well}. 
We set $D\coloneq\mathbb{T}^d$. For any $s>d/2$, define the function spaces
\begin{align}
\mathcal{W}^s & \coloneq H^s(D) \times [H^s(D)]^d,\label{Ws}\\
\f{V}^{s+1} & \coloneq \{\f{j}\in [H^s(D)]^d\colon \nabla\cdot\f{j} \in H^s(D)\}\label{VsPlus1},
\end{align} 
where $H^s(D)$ is the usual Sobolev space of  functions with $s$ square-integrable weak derivatives. 
The model \eqref{ridk} is rewritten in the abstract stochastic PDE form
\begin{equation}
 \label{RidkAbstractRewriting}
   \begin{aligned}
    \!\!\!\!\! \m u(t)&=A\,u(t)\,\m t + F(u(t))\m t 
    +B_{N}(u(t))\,\m W_{\ep}(t), \\
    \!\!\!\!\!  u(0)&=u_0,
    \end{aligned}
\end{equation}
where $u=u_{\ep}=(\rho,\f{j})$, $W_{\ep}\coloneq(0,\f{\xi}_{\ep})$ is a $\mathcal{W}^s$-valued $Q$-Wiener noise (see Subsection \ref{AppNoiseExpansion} for full details), $A$ is the wave-type differential operator given by
\begin{align*}
  A\,u\colon \mathcal{D}(A)\coloneqq H^{s+1}(D)\times\f{V}^{s+1} \subset \mathcal{W}^s&\rightarrow \mathcal{W}^s, \\
  u=(\rho,\f{j}) &\mapsto \pp{-\nabla\cdot\f{j},
  \,-\gamma \,{\f{j}}
  -k_BT\,\nabla{\rho} }, 
\end{align*}
where $F(\rho,\vec j)=[0,-\nabla \potential\,\rho]$, and where the stochastic integrand $B_N$ is given by
\begin{align*}
  B_{N}({\rho},\f{j})(a,\f{b})
  \coloneqq \sigma\, N^{-1/2}\left(0,\, \sqrt{{\rho}}\,b_1,\dots,
  \sqrt{\rho}\,b_d\right)\in \mathbb{R}\times\mathbb{R}^d.
\end{align*}
As \eqref{ridk-rho} is a conservation law, we denote the total mass of the system as
\begin{align}\label{ConservedMass}
\mass \coloneq \int_{\domain}{\rho_0}(\vec x)\m \vec x.
\end{align}

We have the following well-posedness result for \eqref{RidkAbstractRewriting}.

 \begin{proposition}\label[proposition]{prop-pathwise}
Assume the scaling $N\ep^{\theta}=1$, for some $\theta>2d$. For initial data $(\rho_0,\f{j}_0) \in \mathcal{W}^s$ with $\rho_0$ positive and uniformly bounded away from zero, for every $\nu\in(0,1)$ and a suitable $T(\rho_0)>0$, there exists $N(\nu,T)\in\mathbb{N}$, a measurable set $F_{\nu}\subset \Omega$ with $\mathbb{P}(F_{\nu})\geq \nu$, and a unique non-negative process $u$ solving \eqref{RidkAbstractRewriting} on $F_{\nu}$ in a mild solution sense, meaning that 
\begin{align}\label{MildSolution}
  u(t) & = S(t)\,u_0 
  +\int_{0}^{t}{S(t-s)\,F(u(s))\,\emph{\m} s}
  +\int_{0}^{t}{S(t-s)\,
    B_{N}(u(s))\,\emph{\m} W_{\ep}}(s)
  \end{align}
on $F_{\nu}$ and $t\leq T$, for $N\geq N(\nu,T)$, and where $\{S(t)\}_{t\geq 0}$ is the $C_0$-semigroup associated with the operator $A$. Furthermore, \eqref{BoundWsRIDK} holds. 
 
If, in addition, the more restrictive scaling \eqref{ScalingNEpsilon} is satisfied, then $u$ is also a path-wise solution up to time $T=T(\rho_0)$, namely, we have 
\begin{align}\label{PathwiseSmoothedIntegrand}
u(t)=u_0 + \int_{0}^{t}{Au(s)\,\emph{\m} s}+
\int_{0}^{t}{F(u(s))\,\emph{\m} s}+ \int_{0}^{t}{B_{N}(u(s))\,\emph{\m} W_{\ep}(s)}
\end{align}
  on $F_{\nu}$ and $t\leq T$.
\end{proposition} 

We recall the proof of the proposition from  \cite{cornalba2021well} in Subsection \ref{app:pathwise}.

\begin{remark}\label{rem:1}
As discussed in \cite{cornalba2021well}, the process $u$ is defined on the entire probability space, but only solves \eqref{ridk-rho}--\eqref{ridk-j} with high-probability (i.e., on $F_{\nu}$). More precisely, upon modifying the noise factor $B_N$ by replacing the square root function with a smooth $h_\delta \in C^{\lceil d/2 \rceil + 2}$ such that $h_{\delta}(z) = \sqrt{|z|}$ for $|z|\geq \delta$, and performing a truncation on a $\mathcal{W}^s$-sphere with sufficiently large radius $k$ (this modified noise is denoted by $B_{N,\delta}$), then one has a mild solution defined on the whole probability space $\Omega$. On the set $F_{\nu}$, the noises $B_N$ and $B_{N,\delta}$ coincide, and we therefore say that $u$ satisfies the dynamics \eqref{ridk-rho}--\eqref{ridk-j} on the set $F_{\nu}$. The cut-off parameter $\delta$ is chosen so that $\min_{\f{x}\in D}{\rho_0(\f{x})}> \delta$.
\end{remark}

 \begin{remark}\label{JustifyScaling}
The well-posedness of \eqref{RidkAbstractRewriting} in \cite{cornalba2021well} (in terms of mild solutions) is related to the scaling regime $N\ep^\theta = 1$, with constraint $\theta > 2d$: this constraint boils down to the relation
 \begin{align}\label{BoundWsRIDK}
\mean{\|u(t)\|^2_{\mathcal{W}^{s}}} \propto N^{-1}\ep^{-2s-d},
\end{align}
which explicitly links the Sobolev space index $s$ to the critical threshold $\theta_c\coloneq 2s+d$, and the Sobolev embedding requirement $s>d/2$, which allows the embedding into the continuous functions. Form a physics perspective, the constraint $\theta > 2\theta$ implies particle overlap (as each particle's volume is proportional to $\ep^{d}$).

The scaling in \eqref{ScalingNEpsilon} is more restrictive, as we also want to account for: i) standard interpolation error estimates of our DG method, as will become apparent in \cref{BHA}, and ii) a stronger notion of solution (analytically strong solution), see \eqref{PathwiseSmoothedIntegrand} above. 
\end{remark}

\section{discontinuous galerkin framework }\label{GalerkinFramework}

We develop the weak form for the Raviart--Thomas mixed finite-element approximation in space, including the numerical flux.

\subsection{Basic notation}
Consider 
a tesselation $\mathcal{T} _{h}$ of $D= \mathbb{T}^d$ consisting of simplicial elements $K$ (triangles, tetrahedrons,...) with maximum side-length $h$. We denote 
by $\mathcal{E}_h$ the set of facets $e$ (edges of triangles, faces of tetrahedrons...) In dimension $d=1$, the simplicial elements are simply intervals.

For $q\in \mathbb{N}_0$, let polynomials of degree-$q$ on $K\subset D$ be denoted $\mathcal{P}_q(K)$.
For  $q\in \mathbb{N}$, denote the Raviart--Thomas elements of order $q$ (see, for instance, \cite[Chapter 3]{kirby2012common} or \cite{Brezzi1991-rz})  by $\mathcal{RT}_{q}(K)=(\mathcal{P}_{q-1}(K))^d+\vec x\, \mathcal{P}_{q-1}(K)$. Let $H(\mathrm{div}, D)\coloneq\{\f{v}\in [L^2(D)]^d\colon \nabla\cdot \f{v}\in L^2(D)\}$ and $\mathcal{RT}_q\coloneq\{\f{v}\in H(\mathrm{div},D)\colon \f{v}|_K \in \mathcal{RT}_q(K) \;\forall K \in \mathcal{T}_h\}$. Due to existence of the weak divergence, any $\f{v}\in \mathcal{RT}_q$ has continuous normal component across every $e\in\mathcal{E}_h$.
Let $\mathcal{DG}_q=\{v\in L^2(D)\colon v|_K \in \mathcal{P}_q(K)\; \forall K\in \mathcal{T}_h\}$.
We define the function spaces
\begin{align}
V &  \coloneq \left\{ u=(\rho,\vec j)\in \mathcal{W}^{\overline{s}}\colon \int_{\domain}{\rho(\f{x})\m \f{x}}=\mass \right\},\label{SpaceV}\\
V_0 & \coloneq \left\{ u=(\rho,\vec j)\in \mathcal{W}^{\overline{s}}\colon \int_{\domain}{\rho(\f{x})\m \f{x}}=0 \right\}\label{SpaceVZero},
\end{align}
for $\overline{s}$ as in \eqref{ScalingNEpsilon}, and where $\mass$ has beed defined in \eqref{ConservedMass}. Their DG counterparts are denoted by
  \begin{align}
  V_{h} & \coloneq \left\{\rho_h\in \mathcal{DG}_q\colon \int_D \rho_h(\vec x)\,\m\vec x=\mass \right\} \times \mathcal{RT}_{q+1},\label{SpaceVh}\\
 V_{h,0} & \coloneq \left\{\rho_h\in \mathcal{DG}_q\colon \int_D \rho_h(\vec x)\,\m\vec x=0 \right\} \times \mathcal{RT}_{q+1}.\label{SpaceVh0} 
 \end{align}
 
We denote the $L^2$-inner product (respectively, $L^2$-norm) by $(\cdot,\cdot)$ (respectively, by $\|\cdot\|$). Furthermore, we use the notation
\begin{align}\label{WeightedL2Norm}
\ip{(\rho,\f{j}),(\phi,\f{\psi})} \coloneq  k_BT(\rho,\phi) + (\f{j},\f{\psi}).
\end{align}

Finally, throughout the paper, we use the letted ``$\m$'' to indicate the standard It\^o differential.

\subsection{Derivation of weak form}\label{DerivationOfWeakForm}

In order to derive the weak form of~\eqref{ridk}, we multiply \eqref{ridk-rho} by a test function $\phi\colon K\to \reals$  and integrate over an element $K$ with boundary $\partial K$, obtaining
\[
\int_K \phi\,\m \rho \,\m \vec{x}%
= - \int_K \phi\, (\nabla \cdot \vec{j})\,\m\vec{x} \m t
=-\int_{\partial K}\phi \,(\vec{j}\cdot \vec{n})\,\m S\m t+\int_K (\nabla \phi) \cdot \vec{j}\,\m\vec{x}\m t,
\]
where $\vec n$ is the outward-pointing normal on  $K$. 
For piecewise-constant elements, where $\nabla \phi$ vanishes,  conservation of mass holds and
\begin{align*}
\m \int_K \,\rho \,\m\vec{x}
=-\int_{\partial K} (\vec{j}\cdot \vec{n})\,\m S\m t.
\end{align*}
With $\potential =0$ (for simplicity), testing \eqref{ridk-j}  with $\vec{\psi} \colon K\to \reals^d$ and then applying the divergence theorem entails
\begin{align*}
\int_K \vec{\psi} \cdot \m \vec{j}\,\m\vec{x}
&= -\gamma \int_K \vec{\psi}\cdot\vec j \,\m\vec{x}\m t
 -k_B T \int_K \vec{\psi} \cdot\nabla \rho\,\m\vec{x}\m t + \sigma\frac{1}{\sqrt N}\int_K \sqrt{\rho}\,(\vec{\psi}\cdot \,\m \vec{\xi}_\ep)\,\m\vec{x}\\
&= -\gamma \int_K \vec{\psi}\cdot\vec j \,\m\vec{x} \m t
-k_B T\int_{\partial K} (\vec{\psi}\cdot \vec n)\, \rho \,\m S\m t\\
&\quad+k_B T \int_K (\nabla\cdot\vec{\psi}) \,\rho\,\m\vec{x}\m t
  + \sigma\frac{1}{\sqrt N}\int_K \sqrt{\rho}\,(\vec{\psi}\cdot \m\vec{\xi}_\ep)\,\m\vec{x}.
\end{align*}
Under the condition that $(\vec \psi\cdot \vec n) \rho$ is continuous across $\partial K$, we can sum over $K\in\mathcal{T}_h$, drop the null contribution 
$
-k_B T\sum_{K}  \int_{\partial K}(\vec\psi\cdot \vec n) \,{\rho} \,\m S, 
$
and derive the following weak form: for initial condition $(\rho_0,\vec j_0)\in V$, find $(\rho, \vec j)\colon[0,T]\to V$ such that
\begin{align*}
(\phi, \m \rho)
&=-\sum_{K\in\mathcal{T}_h} \int_{\partial K}  \phi\,(\vec{j}\cdot \vec{n})\,\m S\m t+(\nabla \phi, \vec j)\m t,\\
(\vec{\psi},\m \vec{j})
&= -\gamma (\vec{\psi}, \vec j )\m t+k_B T(\nabla\cdot\vec\psi, \rho)\m t + \frac{\sigma}{\sqrt N} \left(\sqrt{\rho}\vec \psi, \m\vec \xi_\ep\right) 
\end{align*}
holds for any $(\phi,\f{\psi})\in V_0$. 
We keep the boundary terms in $\vec j$ as this is useful for the following analysis. 
Including the external potential $\potential$, we write the weak formulation as: find $u=(\rho,\vec j)\colon [0,T]\to V$ such that 
\begin{align}\label{cont_weak_form}
\m \ip{u(t),v}
=a(u(t),v)\m t +\ip{F(u(t)), v}\,\m t + \ip{B_N(u(t)) \m W_{\ep}, v},\qquad%
\forall v=(\phi,\vec \psi)\in V_0,
\end{align}
where we have defined the bilinear form $a$ as
\begin{align}\label{d:16}
  a(u,v) & \coloneq  -k_BT\sum_{K\in\mathcal{T}_h} \int_{\partial K}  \phi\,(\vec{j}\cdot \vec{n})\,\m S+ k_BT ( \nabla \phi, \vec j)\nonumber\\
  & \quad -\gamma (\vec{\psi}, \vec j ) + k_B T(\nabla\cdot\vec\psi, \rho).
\end{align}
Given the regularity of the solution $u$ in \cref{prop-pathwise}, such a solution $u$ also solves the weak formulation \eqref{cont_weak_form}, as detailed in the proof of \cref{prop-pathwise}.

\subsection{Discontinuous Galerkin method}\label{DiscontinuousGalerkinMethod}
For the discontinuous Galerkin method, we approximate $\rho$ and $\vec j$ by discontinuous functions with jump discontinuities on the facets $e\in\mathcal{E}_h$. Their 
 values must be assigned on the facets: in accordance with the notation in \cite{Arnold2000-bk}, we introduce the numerical flux $h_{\rho}^{e,K}$ (respectively, $\vec{h}_j^{e,K}$) to replace $\rho$ (respectively, $\vec j$) on the facet $e$ of the element $K$. These fluxes will be defined in Subsection \ref{sec_fluxes} below.
  Then, the weak form with numerical fluxes becomes: seek $(\rho_h, \vec j_h)\colon [0,T]\to V_h$ such that 
  \begin{subequations}
  \label{weak-formulation}
\begin{align}
(\phi_h, \m \rho_h) 
&=-\sum_{K\in \mathcal{T}_h} \int_{\partial K}  \phi_h\,\vec{h}_{j}^{e,K}\cdot \vec n\,\m S\m t+(\nabla \phi_h, \vec j_h)\m t,\\
(\vec{\psi}_h, \m \vec j_h)
&= -\gamma(\vec{\psi}_h, \vec j_h )\m t-(\nabla \potential\,\rho_h,\vec{\psi}_h)\,\m t -k_B T\sum_{K\in \mathcal{T}_h}  \int_{\partial K}(\vec\psi_h\cdot \vec n)  \,h_{\rho}^{e,K} \,\m S\m t\nonumber\\
& \quad+k_B T\,(\nabla\cdot\vec\psi_h, \rho_h)\m t
  + \frac{\sigma}{\sqrt N} \left(\sqrt{\rho_h}\vec \psi_h, \m \vec \xi_{h,\ep}\right)
\end{align}
\end{subequations}
for all $(\phi_h, \vec \psi_h)\in V_{h,0}$, where $\vec \xi_{h,\ep}$ is a suitable truncation of $\vec \xi_\ep$ to be specified later. 

\subsubsection{Definition of fluxes}\label{sec_fluxes}
We define numerical fluxes by solving the wave equation attained by neglecting the dissipation and noise components of~\eqref{ridk} in one-dimensional cross-sections normal to $e\in \mathcal{E}_h$. (following the method of Godunov~\cite{Toro1999-gg}, intuition in the one dimensional case is discussed in Remark \ref{intuition_fluxes} below). Specifically, we set
\begin{subequations}
\makeatletter
        \def\@currentlabel{Flux}
        \makeatother
        \label{flux}
        \renewcommand{\theequation}{Flux\arabic{equation}}   
\begin{align}
h_{\rho}^{e,K}
&\coloneq\{\rho\}+\frac{1}{2\sqrt{k_B T}}\jump{\vec j},\tag{Flux-$\rho$}\label{FluxRho}\\
 \vec{h}_j^{e,K}&
\coloneq\{\vec j\} + \frac{\sqrt{k_B T}}{2}\jump{\rho},
\tag{Flux-$\f{j}$}
\label{FluxJ}
\end{align}
\end{subequations}
where we have used $\{\cdot\}$ to denote the average value on either side of $e$ and $\jump{\phi}=2\{\phi \,\vec n\}$ or $\jump{\vec\psi}=2\{\vec \psi\cdot\vec n\}$ to denote the jump for scalar or vector quantities. The fluxes are consistent as $h_{\rho}^{e,K}\equiv\rho$ and $\vec{h}_{j}^{e,K}\equiv\vec j$ if $\rho$ and $\vec j$ are continuous (as jumps across edges are null, and average values across edges coincide with the values on the edge). If $K_+, K_-$ share a facet $e$, then
$h_\rho^{e,K_+}=h_\rho^{e,K_-}$ 
and $\vec{h}_j^{e,K_+}=\vec{h}_j^{e,K_-}$ and we may drop the $K$ superscript. 
For $\vec j\in \mathcal{RT}_{q+1}$, the normal components of $\vec j$ is continuous across $e\in \mathcal{E}_h$ and $h_{\rho}^e=\{\rho\}$.

\begin{remark}\label{intuition_fluxes} To give an intuition for the definitions \eqref{flux},
consider the non-dissipative linear part of the noise-free version of \eqref{ridk}, in dimension one, and without boundary conditions, namely
\begin{equation}
\rho_t=-j_x, \qquad j_t=-k_B T \rho_x, \qquad t>0, \,\,x\in\reals.%
\label{w1}
\end{equation}
Equation \eqref{w1} is a wave equation with wave speed $c=\sqrt{k_B T}>0$ in $\rho$, and initial conditions $\rho(0,x)=\rho_0(x)$ and $\rho_t(0,x)=-j_0'(x)$. Its general solution is $\rho(t,x)=A(x-c t) + B(x+c t)$ for $A,B\colon \reals \to \reals$. To match $A,B$ to the initial data, put $A+B=\rho_0$ and $-c A'+c B'=-j_0'$. Then $2 c B'=-j_0'+c\rho_0'$ and $B=(1/2c)( c\rho_0- j_0)$. Similarly, $A=(1/2c) (c\rho_0+ j_0)$.

To derive the flux for a discontinuous Galerkin method, consider initial data with a jump at $x=0$, namely,
\begin{equation}
  \rho_0(x,t)=\begin{cases} \rho_-,& x<0\\ \rho_+, &x\ge 0;\end{cases}\qquad
  j_0(x,t)=\begin{cases} j_-,& x<0,\\ j_+, &x\ge 0,\end{cases}\label{w2}
\end{equation}
for constants $\rho_\pm, j_\pm$. The solution to equation \eqref{w1} is then given by
\begin{align}
  \rho(x,t)\label{rho_1d}
  =\begin{cases} \rho_-,& x<-ct,\\ 
    \frac12(\rho_+ +\rho_-) + \frac{1}{2c}(j_+-j_-), &-ct<x\le ct,\\
    \rho_+, &ct<x,\end{cases}
    \end{align}
    and
    \begin{align}\label{j_1d}
  j(x,t)=\begin{cases} j_-,& x<-ct,\\ \frac12(j_+ + j_-)+\frac{c}{2}(\rho_+ -\rho_-), &-ct<x\le ct,\\ j_+,& ct<x.\end{cases}
  \end{align}
 Since $c=\sqrt{k_B T}$, the analogy between \eqref{rho_1d}--\eqref{j_1d} and \eqref{flux} is now apparent.
\end{remark}
 
\subsection{Weak formulation of \eqref{weak-formulation}} 
Taking once more into account that $\vec \psi_h\in \mathcal{RT}_{q+1}$ leads to $\jump{\vec\psi_h}=0$, the weak formulation \eqref{weak-formulation} reduces to: find $(\rho_h,\f{j}_h )\colon [0,T]\to V_h$  such that
\begin{subequations}
\makeatletter
        \def\@currentlabel{DG}
        \makeatother
        \label{dg_ridk}
        \renewcommand{\theequation}{DG\arabic{equation}}
\begin{align}
(\phi_h, \m \rho_h)\tag{DG-$\rho$}
&=-\sum_{ e\in \mathcal{E}_h} \int_{ e}  \jump{\phi_h}\cdot\vec h_j^{e}\,\m S\m t + (\nabla \phi_h, \vec j_h)\m t,\\
(\vec{\psi}_h, \m \vec j_h )
&= -\gamma (\vec{\psi}_h, \vec j_h )\m t -(\nabla \potential \rho_h,\vec{\psi}_h)\,\m t \nonumber \\
& \quad + k_B T(\nabla\cdot\vec\psi_h, \rho_h)\m t\nonumber  + \frac{\sigma}{\sqrt N} \left(\sqrt{\rho_h}\vec \psi_h, \m\vec \xi_{h,\ep}\right)\tag{DG-$\f{j}$}
\end{align}
\end{subequations}
for all $(\phi_h, \vec \psi_h)\in V_{h,0}$.
Defining the bilinear form $a_h$ with arguments $u_h=(\rho_h,\f{j}_h)$ and $v_h=(\phi_h,\f{\psi}_h)$
\begin{align}\label{ah}
a_h(u_h,v_h) \coloneq & -k_BT\sum_{ e\in \mathcal{E}_h} \int_{ e}  \jump{\phi_h}\cdot\vec h_j^{e}\,\m S + k_BT(\nabla \phi_h, \vec j_h)-\gamma (\vec{\psi}_h, \vec j_h ) \nonumber\\
&  +k_B T(\nabla\cdot\vec\psi_h, \rho_h),
\end{align}
and setting $W_{h,\ep}\coloneq (0,\vec\xi_{h,\ep})$, we can rewrite \eqref{dg_ridk} as follows: find $(\rho_h,\f{j}_h)=u_h\colon [0,T]\to V_h$ such that
\begin{align}\label{DiscreteWeakFormRIDK}
 \m \ip{u_h(t), v_h}%
  =a_h(u_h(t),v_h)\,\m t+\ip{F(u_h),v_h}\,\m t
  + \ip{B_N(u_h) \m W_{h,\ep}, v_h},\qquad%
   \forall v_h\in V_{h,0}.
\end{align}

\subsection{Rewriting \eqref{DiscreteWeakFormRIDK}}\label{DefnAhQh}
Let $Q_h$ be the projection operator on to $V_h$ with respect to the $\langle\cdot,\cdot\rangle$-inner product defined in \eqref{WeightedL2Norm} (explicitly, if $z\in L^2\times [L^2]^d$, then $\langle Q_hz,v_h\rangle=\langle z,v_h\rangle$ for all $v_h\in V_h$). Furthermore, for $z=(\rho,\vec j)\in \mathcal{W}^{\overline{s}} \cup \{\mathcal{DG}_q \times \mathcal{RT}_{q+1}\}$, we define $A_hz$ as the unique element of $V_{h,0}$ such that $\langle A_hu,v_h\rangle=a_h(u,v_h)$ for all $v_h\in V_{h,0}$.

Taking these definitions into account, as well as the smoothing of the noise integrand (i.e., using $B_{N\delta}$ instead of $B_N$), equation \eqref{DiscreteWeakFormRIDK} can be seen as the variational formulation of the abstract equation 
\begin{equation}
 \label{DiscreteRidkAbstractRewriting}
   \begin{aligned}
    \!\!\!\!\! \m u_h(t)&=A_hu_h(t)\,\m t +Q_hF(u_h)\,\m t
    +Q_hB_{N,\delta}(u_h(t))\,\m W_{h,\ep}(t), \\
    \!\!\!\!\!  u_h(0)&=u_{h,0},
    \end{aligned}
\end{equation}
where $u_{h,0}$ will be chosen below in \cref{ErrorEquationStochastic}. The well-posedness of \eqref{DiscreteRidkAbstractRewriting} is readily settled since the it is an SDE with smooth coefficients. 

\section{properties of linear setting}\label{linear_setting}

We prove relevant properties related to the bilinear forms $a$ (see \eqref{d:16}) by $a_h$ (see and \eqref{ah}). Specifically, we discuss:
\begin{itemize}
\item a suitable \emph{inf-sup} condition (also known as \emph{LBB} condition \cite{boffi2013mixed}), see \cref{nearly_coercive}, and
\item continuity, see \cref{bdd_ah}.
\end{itemize}

First, we define two useful norms, namely
\begin{align}\label{DaggerNorm}
    \norm{(\rho,\vec j)}^2_\dagger \coloneq 
    \begin{cases}\displaystyle
  \sum_{e\in \mathcal{E}_h}    \norm{\jump{\rho}}_{L^2(e)}^2 
      + \norm{\vec j}^2 + \norm{\nabla \cdot \vec j}^2,&\mbox{if }q=0,\\[1.1em]  
      \norm{\nabla \rho}^2 
      + \norm{\vec j}^2 + \norm{\nabla \cdot \vec j}^2,&\mbox{if }q>0,
    \end{cases}
\end{align}
and 
\begin{align}\label{StarNorm}
     \norm{(\rho,\vec j)}^2_\star \coloneq \norm{ \rho}^2+\sum_{ e\in \mathcal{E}_h} \norm{\jump{\rho}}^2_{L^2( e)} + \norm{\vec j}^2,
     \end{align} 
where we recall that $\|\cdot\|$ is the standard $L^2$ norm.

\begin{lemma}[\emph{inf-sup} condition for bilinear forms $a$ and $a_h$]\label[lemma]{nearly_coercive} We have the following statements:
  \begin{enumerate}[i)]
  \item 
  There exists a constant $C>0$  such that 
    \begin{align*}
      \inf_{e_h=(\rho_h,\vec j_h)\in V_{h,0}} \sup_{v_h \in V_{h,0}}
      \frac{\left|a_h(e_h,v_h)\right|}{\norm{v_h}_\star  \norm{e_h}_\star } %
      \ge C,
      \end{align*}
     where $V_{h,0}$ is given in \eqref{SpaceVh0}, and $\|\cdot\|_\star$ is defined in \eqref{StarNorm}.
     \item 
     There exists a constant $C>0$  such that 
     \begin{align*}
      \inf_{e_h=(\rho_h, \vec j_h)\in V_{h,0}}
   \sup_{v_h\in V_{h,0}} \frac{\left|a(e_h, v_h)\right|}{\norm{v_h} \norm{e_h}}\ge C,
    \end{align*}
    where $   \norm{e_h}^2 \coloneq \norm{ \rho_h}^2 + \norm{\vec j_h}^2$.
     \end{enumerate}
  \end{lemma}

\begin{lemma}[continuity of bilinear forms]\label[lemma]{bdd_ah}
  \begin{enumerate}[i)]
  \item Suppose that $u_h=(\rho_h,\vec j_h)\in L^2(D)\times H(\mathrm{div}, D)$ and that $\rho_h$ is piecewise constant (thus choosing $q=0$). Then we have, for a constant $C>0$,
    \[
      \abs{a_h(u_h,v_h)}%
      \le C\norm{u_h}_\dagger\,\norm{v_h}_\star,\qquad \forall v_h\in V_{h,0},
      \]
      where $\norm{\cdot}_\star$ (respectively, $\norm{\cdot}_{\dagger}$) is defined in \eqref{StarNorm} (respectively, \eqref{DaggerNorm}).
        \item Fix $q>0$, and take $u_h=(\rho_h,\vec j_h)\in \{\mathcal{DG}_q \cap C^0(D)\} \times   \mathcal{RT}_{q+1} $. 
        Then there exists $C>0$ such that
     \[
      \abs{a(u_h,v_h)}%
      \le C\norm{u_h}_\dagger\,\norm{v_h},\qquad \forall v_h\in V_{h,0}.
      \] 
   \end{enumerate}
  \end{lemma}

  \begin{remark}
As will become apparent from the proof Lemma \ref{nearly_coercive}, the validity of this inf-sup condition follows from -- and further justifies -- the definition of the numerical fluxes \eqref{flux}. 
\end{remark}

\begin{proof}[Proof of Lemma \ref{nearly_coercive}] \emph{Part i)}. As the bilinear form $a_h$ is not coercive, we will need to make a special choice of $v_h$ to prove the result. 
  For $e_h=(\rho_h,\vec j_h)$ and $v_h=(\phi_h, \vec \psi_h)\in V_{h,0}$, the bilinear form $a_h(e_h, v_h)$ reads
  \begin{gather}
 \begin{split}   
  a_h(e_h, v_h)
   &=-k_BT\sum_{ e\in \mathcal{E}_h} \int_{ e}  \jump{\phi_h}\cdot \vec h_j^e \,\m S+k_BT (\nabla \phi_h, \vec j_h)\\
   &\quad -\gamma (\vec{\psi_h}, \vec j_h) + k_B T(\nabla\cdot\vec\psi_h, \rho_h).
   \end{split}\label{d:14}
  \end{gather}

Now we define $v_h$ as a special perturbation of $u_h$. More precisely,
 we put  $\phi_h=\rho_h$ and $\vec\psi_h=\vec j_h+ \vec \kappa_h$ 
 for some $\f{\kappa}_h\in\mathcal{RT}_{q+1}$ satisfying
 \begin{align}\label{DivergenceProblemKappa}
 -\nabla \cdot \vec \kappa_h=\eta\,\rho_h \quad\mbox{ on }D
 \end{align} 
 (note the solvability condition $\int_D \nabla\cdot \f{\kappa}_h\,\m \vec x=0$ for a continuous field $\vec \kappa_h$ on a periodic domain holds as $(\rho_h,\vec j_h)\in V_{h,0}$), where $\eta>0$ is to be specified. 
 
The equation \eqref{DivergenceProblemKappa} is, of course, underdetermined. We look for a solution $\f{\kappa}_h$ such that
\begin{align}\label{EmbeddingKappa}
\|\f{\kappa}_h\| \leq C\|\eta \rho_h \|,
\end{align}
is satisfied, with $C>0$ being some constant.

Since the divergence operator maps $\mathcal{RT}_{q+1}(K)$ onto $\mathcal{P}_q(K)$, the equation \eqref{DivergenceProblemKappa} admits at least a solution.
 If we in addition demand that $\vec \kappa_h$ is curl-free, then  $\vec{\kappa}_h \in\mathcal{RT}_{q+1}$ is uniquely defined (by the discrete Helmholz decomposition~\cite{Arnold2000-sz}). 
This is the minimum $L^2(D)$ solution to equation \ref{DivergenceProblemKappa} in $\mathcal{RT}_{q+1}$ and by uniqueness also in $H(\mathrm{div}, D)$: in particular, there exists a solution in $H^{1}(D)$ that satisfies \cref{EmbeddingKappa} (see \cite{BourgainBrezis2003}).

 The extra terms due to $\vec\kappa_h$ (i.e., the terms making up the difference $a_h(u_h,v_h-u_h)$) amount to
  \begin{align}\label{eq_1000}
 & -\gamma (\vec \kappa_h, \vec j_h )
- \eta k_B T \|\rho_h\|^2.
  \end{align}
  Furthermore, we have the identity
   \begin{align}\label{a_u_u}
 a_h(u_h,u_h)
  =
  -\gamma (\vec{j}_h, \vec j_h)
      -\sum_{ e\in \mathcal{E}_h}\int_{ e}  \frac{(k_B T)^{3/2}}{2}\jump{\rho_h}^2\,\emph{d}S.  
  \end{align}
  The identity \eqref{a_u_u} is shown as follows: The flux choice \eqref{flux} and the continuity property of $\vec j_h$ give
 \[   
      \jump{\rho_h} \cdot\vec{h}_j^e=\jump{\rho_h}\cdot \left(\{\vec j_h\} + \frac{\sqrt{k_B T}}{2}\jump{\rho_h}\right)
      =\jump{\rho_h}\cdot\{\vec j_h\} + \frac{\sqrt{k_B T}}{2}\jump{\rho_h}^2.
      \]
    The divergence theorem implies
  \[
    \int_K \nabla \rho_h\cdot \vec j_h\,\m\vec x
    +\int_K (\nabla\cdot\vec j_h) \rho_h\,\m\vec x
    =\int_{\partial K} \rho_h\, \vec j_h\cdot \vec n\,\m S.
    \]
    Taking $\rho_h=\phi_h$ and $\f{j}_h=\f{\psi}_h$ in \eqref{ah} entails
    \begin{align*}
    a_h(u_h,u_h) & = -\gamma (\vec{j}_h,\vec j_h )\\
&\quad      -\sum_{ e\in \mathcal{E}_h}\int_{ e}\pp{k_B T \jump{\rho_h} \cdot\vec j_h + \frac{(k_B T)^{3/2}}{2}\jump{\rho_h}^2  + k_B T \vec j_h\cdot \jump {\rho_h}}\,\m S\\
      &= -\gamma (\vec{j}_h,\vec j_h)
      -\sum_{ e\in\mathcal{E}_h}\int_{ e}  \frac{(k_B T)^{3/2}}{2}\jump{\rho_h}\cdot\jump{\rho_h}\,\m S,
    \end{align*}
and \eqref{a_u_u} is settled.
  
    Combining \eqref{a_u_u} and \eqref{eq_1000}, we deduce
  \begin{align*}
    a_h(e_h,v_h)\le & -\gamma \norm{\vec j_h}^2 - \gamma (  \vec \kappa_h, \vec j_h)\\
    & - \sum_{\in\mathcal{E}_h} \int_{e} \frac{(k_B T)^{3/2}}{2}\jump{\rho_h}^2\,\m S -\eta k_B T   \|\rho_h\|^2.
  \end{align*}  
  The Cauchy--Schwartz inequality and the bound \eqref{EmbeddingKappa} promptly give
  \begin{align*}
  -\gamma \norm{\vec j_h}^2 - \gamma \int_D \vec \kappa_h \cdot \vec j_h\m \vec x 
 &  \le  -\gamma \norm{\vec j_h}^2 + \frac 12\gamma \norm{\vec\kappa_h}^2 + \frac 12 \gamma\norm{\vec j_h}^2\\
  & \leq -\frac 12\gamma \norm{\vec j_h}^2 + \frac 12 \gamma K^2 \eta^2 \norm{\rho_h}^2.
  \end{align*}
We conclude that 
\begin{align*}
a_h(e_h,v_h)\le & -\frac 12 \gamma\norm{\vec j_h}^2+\left(-k_B T \eta  + \frac 12 \gamma  K^2  \eta^2\right)  \norm{\rho_h}^2\\
&-\sum_{e\in\mathcal{E}_h}\int_{ e} \frac 12 (k_B T)^{3/2} \jump{\rho_h}\cdot\jump{\rho_h}\,\m S.
\end{align*}
Set $\eta \coloneq k_B T/(\gamma K^2)$. Then,  $-C_1\coloneq -k_B T \eta  + \frac 12 \gamma  K^2  \eta^2<0$. Thus,
\[
|a_h(e_h,v_h)|\ge
C_1  \norm{\rho_h}^2 + C\sum_{e\in \mathcal{E}_h} \norm{\jump{\rho_h}}^2_{L^2( e)} + \frac 12\gamma\norm{\vec j_h}^2.  
\]
Also, 
\begin{align*}
\norm{v_h}_\star^2
& = \norm{\rho_h }^2 
+ \sum_{e\in\mathcal{E}_h} \norm{\jump{\rho_h}}^2_{L^2( e)}
+ \norm{\vec j_h+\vec \kappa_h}^2\\
& \le (1+K^2 \eta^2)  \norm{\rho_h}^2 + \sum_{e\in\mathcal{E}_h} \norm{\jump{\rho_h}}^2_{L^2( e)}
 + \norm{\vec j_h}^2.
\end{align*}
Putting all together, we obtain
\[
  \sup_{v_h\in V_{h,0}}\frac{|a_h(e_h,v_h)|}{\norm{v_h}_\star}\ge \frac{C_1  \norm{\rho_h}^2
  + C\sum_{e\in\mathcal{E}_h} \norm{\jump{\rho_h}}^2_{L^2( e)} +\frac 12 \gamma \norm{\vec j_h}^2}{(1+K^2\eta^2)\norm{u_h}_\star}\ge C\norm{e_h}_\star,\]
and \emph{Part i)} is settled. As for \emph{Part ii)}, the only difference is the lack of the boundary terms for $\rho_h$ (we consider $a_h$ instead of $a$). Therefore, we get the same result, only with $\|\cdot\|$ replacing $\|\cdot\|_{\star}$.
\end{proof}

 \begin{proof}[Proof of Lemma \ref{bdd_ah}]
    
For \emph{Part i)}, $q=0$ and $\phi_h$ is piecewise constant.
Therefore, we obtain
  \begin{align*}
  a_h(u_h, v_h)
   &=-k_BT\sum_{ e\in\mathcal{E}_h} \int_{ e}  \jump{\phi_h}\cdot\vec{h}_j^{ e} \,\m S - \gamma ( \vec{\psi_h}, \vec j_h)\\
   & \quad  -k_B T\sum_{e}  \int_{\partial e} 2\jump{\vec\psi_h}\,h_\rho^e \,\m S+k_B T(\nabla\cdot\vec\psi_h, \rho_h)\\
   &=-k_BT\sum_{ e\in\mathcal{E}_h} \int_{ e}  \jump{\phi_h}\cdot\{\vec j_h\}\,\m S-(k_BT)^{3/2}\sum_{ e\in\mathcal{E}_h} \int_{ e}  \jump{\phi_h}\cdot\jump{\rho_h}\,\m S\\
   & \quad -\gamma (\vec{\psi_h}, \vec j_h) + k_B T (\nabla\cdot\vec\psi_h, \rho_h),
   \end{align*}
   where we have used the continuity of $\vec j_h, \vec\psi_h$. As $\phi_h$ is piecewise constant, applying the divergence theorem in the first term in the right-hand-side above yields
   \begin{align*}  
    a_h(u_h, v_h)
   &=- k_BT (\nabla \cdot \vec j_h, \phi_h )  -(k_BT)^{3/2}\sum_{ e\in\mathcal{E}_h} \int_{ e}  \jump{\phi_h}\,\jump{\rho_h}\,\m S\\
   & \quad - \gamma (\vec\psi_h , \vec j_h )
     + k_B T (\nabla \cdot \vec \psi_h, \rho_h).
   \end{align*}
 The Cauchy--Schwartz inequality promptly gives
    \begin{align*}
     |a_h(u_h, v_h)| & \leq C(k_BT,\gamma)\pp{\|\nabla\cdot \vec j_h\|\|\phi_h\|  +\sum_{\smash{e\in \mathcal{E}_h
     }}\norm{\jump{\phi_h}}_{L^2(e)} \norm{\jump{\rho_h}}_{L^2(e)}+ \|\vec\psi_h\|\|\vec j_h\|}\\
     & \leq \norm{u_h}_\dagger\,\norm{v_h}_\star,
    \end{align*}
    and \emph{Point i)} is settled. As for \emph{Point ii)}, we use the divergence theorem, \eqref{d:16}, and the continuity of $\rho_h$ to obtain
         \begin{align*}
    a(u_h, v_h)
   & = - k_BT (\nabla \cdot \vec j_h, \phi_h ) - \gamma (\vec\psi_h , \vec j_h) - k_BT (\vec \psi_h, \nabla \rho_h),
   \end{align*}
   and the proof is concluded by application of the Cauchy--Schwarz inequality.
    \end{proof}

\section{Linear error analysis}\label{ErrorAnalysisLinearPDE}

For the zero potential ($\potential=0$) and deterministic ($\sigma=0$) problem, we quantify the error arising from approximating $a$ (see \eqref{d:16}) by $a_h$ (see and \eqref{ah}).
Let $\overline{s}$ be as in \eqref{ScalingNEpsilon}. For any $z=(\rho,\vec j)\in\mathcal{W}^{\overline{s}-1}$ with mass $\mass$ (i.e., $\int_{\domain}{\rho}(\vec x)\m \vec x = \mass$), we define the Ritz--Galerkin projection $R_hz$ as the unique element of $V_{h}$ such that
\[a(R_h z, v_h )=a(z,v_h),\qquad \forall v_h\in V_{h,0}.\]
Similarly, we define the projection $\tilde{R}_hz$ as the unique element of $V_{h}$ such that
\[a_h(\tilde{R}_h z, v_h )=a_h(z,v_h),\qquad \forall v_h\in V_{h,0}.\]
For $q\geq 0$, and $z=(\rho,\vec j)$ being sufficiently regular, let $\mathcal{I}_q z$ denote the canonical interpolation operator defined component-wise on $\mathcal{CG}_q$ and $\mathcal{RT}_{q+1}$. Furthermore, let $\mathcal{I}_{q,m} z$ be the same as $\mathcal{I}_q z$, but with the first component shifted so as to have the same mass as $\rho$.

The main result of this section is the following.

\begin{lemma}[Ritz--Galerkin error]\label[lemma]{lemma:rge}
For $z=(\rho,\vec j)\in C^{0}\cap\mathcal{W}^{\overline{s}-1}$ with mass $\mass$, there exists a constant $C>0$ such that
 \begin{align}\label{RhBounds}
& \max\left\{\norm{z-\tilde{R}_h z};\norm{z-R_h z}\right\}  \le 
 \begin{cases}
  C h^{1/2} \norm{z}_{\mathcal{W}^{2}},\qquad &\mbox{ if } q=0,\\[1.1em]
  C h^{q} \norm{z}_{\mathcal{W}^{q+2}}, \qquad & \mbox{ if } q> 0.
 \end{cases}
\end{align}
\end{lemma}

The proof of \cref{lemma:rge} relies on \cref{nearly_coercive}, \cref{bdd_ah}, as well as on the following standard interpolation estimate for the norm $\|\cdot\|_{\dagger}$ introduced in \eqref{DaggerNorm}.

\begin{theorem}[Interpolation error]\label[theorem]{BHA} 
  Assume that $\mathcal{T}_h$ is shape regular with mesh width $h$.  There exists a constant $c>0$ such that 
    \begin{enumerate}[i)]
      \item   for $q=0$, 
      \[\norm{z-\mathcal I_0 z}_\dagger \leq c\, h^{1/2}\, \norm{\rho}_{H^1(D)}
      +c\,h\,\bp{\norm{\vec j}_{H^{1}(D)}+\norm{\nabla\cdot\vec j}_{H^1(D)}},
      \]
      \item  for $\mathbb{N}\ni q\geq 1$, 
      \[
      \norm{z-\mathcal I_q z}_\dagger %
      \le c\,h^{q} \,\norm{\rho}_{H^{q+1}(D)}+c\,h^{q+1}\,\bp{\norm{\vec j}_{H^{q+1}(D)}+\norm{\nabla\cdot\vec j}_{H^{q+1}(D)}},
      \]
    \end{enumerate} 
    for all $z=(\rho,\vec j)\in \mathcal{W}^{q+2}$. 
  \end{theorem}
\begin{proof}[Proof of \cref{BHA}] 
Consider $q>0$. Let $\pi_q$ denote the projection onto the degree-$q$ piecewise continuous Lagrange interpolant of degree $q$ on $\mathcal T_h$ (with a standard set of interpolation points \cite[Definition 3.3]{Logg2012-xi}). For $\rho\in H^{q+1}(D)$, standard approximation theory  gives that
  \[
  \norm{\rho-\pi_q \rho}_{H^1(D)}%
  \le C\,h^{q} \,\norm{\rho}_{H^{q+1}(D)}
  \]
(e.g.,  \cite[Eq. (3.12)]{Logg2012-xi}). Hence, $\norm{\nabla \rho - \nabla \pi_q \rho}_{L^2(D)} \le c\, h^q\, \norm{\rho}_{H^{q+1}(D)}$. 

Consider $q=0$. Let $\pi_0$ denote the projection onto the piecewise constant interpolant on $\mathcal T_h$. Let $K\in \mathcal{T}_h$ have boundary $\partial K$ (which consists of a fixed number of $e\in \mathcal{E}_h$). We obtain
  \begin{align*}
  \norm{\jump{\rho-\pi_0\rho}}_{L^2(\partial K)}&
  \leq C \norm{\rho-\pi_0 \rho}_{L^2(K)}^{1/2}\,\norm{\rho-\pi_0 \rho }_{ H^{1}(K)}^{1/2}\\
  &
  \leq C h^{1/2}\,\norm{\rho }_{ H^{1}(K)}^{1/2},
 \end{align*}
 where the first inequality follows from \cite[T1.6.6]{Brenner2008-jg}.
Summing over the facets $e\in \mathcal{E}_h$ gives
\begin{align*}
\sum_{e\in \mathcal{E}_h}  \norm{\jump{\rho-\pi_0\rho}}^2_{L^2(e)}&
  \leq C \sum_{K\in \mathcal{T}_h}  h\,\norm{\rho}^2_{H^1(K)} 
  = Ch\,\norm{\rho }^2_{H^1(D)}.  
\end{align*}

We now turn to terms in $\norm{\cdot}_\dagger$ that involve $\vec j$.
Let $\Pi_{\mathcal{RT}} \vec j$ denote the canonical interpolant in $\mathcal{RT}_{q+1}$ (this is the $\vec j$-component of $\mathcal{I}_q z$). 
Then,
$$
\norm{ \vec j - \Pi_{\mathcal{RT}} \vec j}_{L^2(D)} 
\le C\, h^{q+1}\, \norm{\vec j}_{H^{q+1}(D)}$$
by standard approximation theory for Raviart--Thomas elements (e.g., \cite[Chapter III, Proposition 3.6 with $k=q$]{Brezzi1991-rz}). 
The divergence operator commutes with interpolation in the  sense  that $\nabla \cdot \Pi_{\mathcal{RT}} \vec j$ equals the projection of $\nabla \cdot \vec j$ onto  $\mathcal{DG}_q$. This leads to
\begin{align*}
  \norm{ \nabla\cdot \vec j - \nabla\cdot \Pi_{\mathcal{RT}} \vec j}_{L^2(D)}
  & \le C\, h^{q+1}\, \norm{\nabla\cdot \vec j}_{H^{q+1}(D)}
\end{align*}
(e.g., \cite[Chapter III, Proposition 3.8]{Brezzi1991-rz}). 
Due to the definition of $\norm{\cdot}_\dagger$, this completes the proof.
\end{proof}

\begin{proof}[Proof of \cref{lemma:rge}]
We first treat the case $q>0$. 
     For any $y_h\in V_{h}$, we have $R_h z -y_h \in V_{h,0}$ and, by the inf-sup condition in \cref{nearly_coercive}(ii), 
    \begin{align*}
    \norm{R_h z-y_h} \le & \sup_{w_h\in V_{h,0}} \frac{\abs{a(R_h z-y_h, w_h)}}{\norm{w_h}}.
  \end{align*}
  By definition of $R_h$, it holds $a(R_h z, w_h)=a(z,w_h)$ for all $w_h\in V_{h,0}$. Hence,
  \begin{align*}
    \norm{R_h z-y_h}
    \le & \sup_{w_h\in V_{h,0}} \frac{|a(z-y_h,w_h)|}{\norm{w_h}}.
  \end{align*}
  We now choose $y_h\coloneq \mathcal I_{q,m} z$.
   Since $z\in C^0$ and $y_h\in C(D)\times \mathcal{RT}_{q+1}$, we can exploit the continuity of $a$ (\cref{bdd_ah}) and deduce
  \begin{align}\label{d:100}
    \norm{R_h z- \mathcal I_{q,m} z } 
    \le \norm{z- \mathcal I_{q,m} z}_\dagger  \sup_{w_h\in V_{h,0}}\frac{\norm{w_h}}{\norm{w_h}} = \norm{z- \mathcal I_{q,m} z}_\dagger.
  \end{align}
  We deduce that
  \begin{align}\label{L2NormR_hProjection}
\norm{z-R_h z} &\le \norm{z- \mathcal{I}_{q,m} z}+\norm{R_h z - \mathcal{I}_{q,m} z} \nonumber \\
&  \stackrel{\mathclap{\eqref{d:100}}}{\le} \,\, \norm{z- \mathcal{I}_{q,m} z}+\norm{z- \mathcal{I}_{q,m} z}_\dagger.
\end{align}
It is easy to see that the difference of $\mathcal{I}_q z$ and $\mathcal{I}_{q,m} z$ is bounded by the right-hand-side of \eqref{RhBounds}.
All is left to do is apply \cref{BHA} to achieve the desired estimate. The same proof also applies when $R_h$ is replaced by $\tilde{R}_h$.

In the case $q=0$, we define $R_h$ using the $a_h$ form (so $R_h$ and $\tilde{R}_h$ coincide in this particular case): in this case, the inf-sup condition as described in \cref{nearly_coercive}(i) holds with respect to the $\norm{\cdot}_\star$-norm \eqref{StarNorm},  which equals the $L^2$-norm plus the norm of jumps in $\rho$ on facets. The continuity described in \cref{bdd_ah}(i) now gives, following the above argument,
\begin{align}\label{d:100a}
  \norm{R_h z- \mathcal I_{0,m} z }_\star
  \le \norm{z- \mathcal I_{0,m} z}_\dagger \sup_{w_h\in V_{h,0}} \frac{\norm{w_h}_\star}{\norm{w_h}_\star} =  \norm{z- \mathcal I_{0,m} z}_\dagger,
\end{align}
where the $\dagger$-norm \eqref{DaggerNorm} is defined in terms of jumps of $\rho$ rather than $\nabla \rho$. This leads to 
\begin{align*}
  \norm{z-R_h z}_\star &\le \norm{z- \mathcal I_{0,m} z}_\star+\norm{R_h z - \mathcal I_{0,m} z}_\star\\
  & \stackrel{\mathclap{\eqref{d:100a}}}{\le} \,\,\, \norm{z- \mathcal I_{0,m} z}_\star+\norm{z- \mathcal I_{0,m} z}_\dagger \le \norm{z- \mathcal I_{0,m} z} + \norm{z- \mathcal I_{0,m} z}_\dagger.
  \end{align*}
 Once again, the difference of $\mathcal{I}_q z$ and $\mathcal{I}_{q,m} z$ is trivially bounded as per the right-hand-side of \eqref{RhBounds}.
Using \cref{BHA} once again completes the proof.
\end{proof}

\begin{remark}\label{RemarkGeneralMass}
The definitions of $R_h,\tilde{R}_h$, as well as the statement of \cref{lemma:rge} can be modified in a straightforward way to allow for a constraint with arbitrary mass. However, we prefer to stick to a notation which reflects the fact that \eqref{ridk} conserves mass.
\end{remark}

\section{error analysis for ridk}\label{ErrorAnalysisRIDK}

We now turn to the error analysis of the DG approximation for  \eqref{ridk}. We work with the abstract systems \eqref{PathwiseSmoothedIntegrand} and \eqref{DiscreteRidkAbstractRewriting} for the RIDK solution $u$ and its semi-discrete DG approximation $u_h$. 

   \begin{lemma}[error equation]\label[lemma]{ErrorEquationStochastic}
    Set $e_h \coloneq R_h u - u_h$, where $u$ (respectively, $u_h$) solves \eqref{PathwiseSmoothedIntegrand} (respectively, \eqref{DiscreteRidkAbstractRewriting}) with initial datum $u_0$ (respectively, $u_{h,0}\coloneq R_hu_0$). 
    Then 
   \begin{align}\label{ErrorEquation}
     e_h(t) &= \int_0^t e^{A_h (t-s)}\,\emph{\m}(Q_h(R_h -I)u) + \int_0^t e^{A_h (t-s)}\, A_h(I-R_h)\,u\,\emph{\m} s \nonumber\\
     &\quad+ \int_0^t e^{A_h (t-s)}\, Q_h(F(u_h)-F(u))\,\emph{\m} s \nonumber\\
    &\quad+ \int_0^t e^{A_h (t-s)}\, Q_h(B_{N,\delta}(u_h)-B_{N,\delta}(u))\,\emph{\m} W_{h,\ep}(s) \nonumber\\
    & \quad + \int_0^t e^{A_h (t-s)}\, Q_hB_{N,\delta}(u)\,(\emph{\m}W_{\ep}-\emph{\m}W_{h,\ep})(s).
   \end{align}
 \end{lemma}

\begin{proof}
  Set $\mathcal L u\,\m s\coloneq \m u-A u\,\m s$ and $\mathcal L_h u\,\m s\coloneq \m u-A_h u\,\m s$. 
  We obtain
  \begin{align*}
  Q_h \mathcal {L} u\,\m s - \mathcal{L}_h u_h\,\m s & = Q_h ( F(u) - F(u_h))\,\m s+Q_h(B_{N,\delta}(u)-B_{N,\delta}(u_h))\,\m W_{h,\ep}(s)\\
  & \quad + Q_hB_{N,\delta}(u)(\m W_\ep-\m W_{h,\ep})(s) \eqcolon \m \eta.  
  \end{align*}
  This implies that $\mathcal L_h u_h \,\m s= Q_h \mathcal L u\,\m s-\m \eta$ and
  \[\mathcal{L}_h (R_h u - u_h)\,\m s%
  =\mathcal L_h R_h u \,\m s- Q_h\mathcal L u\,\m s+\m \eta.
  \]
Substituting for $\mathcal{L}_h$ and $\mathcal{L}$, we obtain 
\[\mathcal{L}_h (R_h u - u_h)\,\m s%
  =\m (R_h u - Q_h u) + (Q_h A u - A_h R_h u)\,\m s +\m\eta.
  \] 
 Moreover, we have the consistency equality 
 \begin{align}\label{consistency}
 Q_h A u = A_h u.
 \end{align} 
 To show \eqref{consistency}, it is sufficient to show that $\langle Q_h A u, v_h\rangle= \langle A u, v_h\rangle= \langle A_h u, v_h\rangle$ for all $v_h \in V_{h,0}$. Since $u$ is continuous, $h_{\rho}^{e,T}=\{\rho\}$ and $\vec{\jmath}_{h,e}\cdot \vec n_+=\{\vec j\}\cdot\vec n_+$, so the forms agree, and \eqref{consistency} is settled. 
 
 Combining \eqref{consistency} with the fact that $Q_h R_h=R_h$ and that 
 and $A_h\tilde{R}_h = A_h$, we have shown that
  \[\mathcal{L}_h (R_h u-u_h)\,\m s=  \m(Q_h(R_h -I)u) + A_h(\tilde{R}_h-R_h)\,u\m s + \m \eta.
  \]  
  Let $e_h=R_h u - u_h$. By definition of $\mathcal{L}_h$, we see that $\m e_h - A_h e_h\,\m s= \tau_h\,\m s+\m\eta$ for $\m\tau _h = \m (Q_h(R_h -I) u) + A_h(\tilde{R}_h-R_h)\,u\m s$. Putting everything together gives \eqref{ErrorEquation}.
  \end{proof}

 \begin{proposition}[error bound]\label[proposition]{ErrorBoundStochastic}
  Let the assumptions of \cref{prop-pathwise} and \cref{ErrorEquationStochastic} be satisfied. Assume the validity of the scaling \eqref{ScalingNEpsilon}, which is 
  \begin{align*}
N\ep^{\theta} = 1, \qquad \theta \geq 2\overline{s}+ d,\qquad\mbox{for some }\overline{s} > \max\{d/2+1; q+3\},
\end{align*}
where $q$ is the order of the DG discretisation. 
  Suppose that $u_0\in \mathcal{W}^{\overline{s}}$, and that $F$ is Lipschitz continuous with respect to the $\mathcal{W}^{\overline{s}}$-norm. 
  Finally, define $W_{h,\ep}$ as the truncation of the noise in \eqref{eq:unreg-noise} over the index set 
  \begin{align}\label{TruncationIndexSet}
  \{\f{j}\in \mathbb{Z}^d\colon |\f{j}|_1 \leq \ep^{-1}|\ln(h^{2\tilde{q}})|\},
  \end{align}
   where $\tilde{q}$ is as in \eqref{ConvergenceOrderSquared}.
     Then we have the estimate
     \begin{align}
     \sup_{0\le t\le T}\label{ErrorBound}
    & \mean{\norm{R_hu(t) - u_h(t) }^2_{L^2(\mathbb{T}^d)}}  \le C(\delta,T,d) \left\{1+ \mean{\|u_0\|^2_{\mathcal{W}^{\overline{s}} }}\right\} e^{C_2\left(\mathcal{V},T\right)}\;h^{2\tilde{q}},
     \end{align}
      \end{proposition}
  \begin{proof}
Lemma \ref{ErrorEquationStochastic} and \eqref{e:1000b} give
   \begin{align*}
     \mean{\norm{e_h(t)}^2} & \leq C\left\{\mean{\norm{\int_0^t e^{A_h (t-s)}\,Q_h(R_h -I)\m u}^2}\right.\\
    & \quad\quad + \mean{\norm{\int_0^t e^{A_h (t-s)}\, A_h(\tilde{R}_h-R_h)\,u\,\m s}^2} \\
     &\quad\quad + \mean{\norm{\int_0^t e^{A_h (t-s)}\, Q_h(F(u_h)-F(u))\,\m s}^2}\\
    &\quad\quad + \mean{\norm{\int_0^t e^{A_h (t-s)}\, Q_h(B_{N,\delta}(u_h)-B_{N,\delta}(u))\,\m W_\ep(s)}^2} \\
    &\quad\quad + \left.\mean{\norm{\int_0^t e^{A_h (t-s)}\, Q_hB_{N,\delta}(u)\,(\m W_\ep-\m W_{h,\ep})(s)}^2}\right\} \eqcolon  \sum_{i=1}^{5}{T_i}.
   \end{align*}
    We estimate the five terms separately.
  
  \emph{Term $T_1$}. We use \cref{lemma:rge}, \cite[Proposition 4.36]{Da-Prato2014a}, as well as \eqref{BoundWsRIDK} to write
  \begin{align*}
  T_1 & \leq \int_0^t \mean{\norm{e^{A_h (t-s)}\,Q_h(R_h -I)Au}^2} \m s \\
  & \quad + \mean{\norm{\int_0^t e^{A_h (t-s)}\,Q_h(R_h -I)B_{N,\delta}(u)\m W_\ep(s)}^2}\\
  & \leq  Ch^{2\tilde{q}}\int_0^t\mean{\|Au\|^2_{\mathcal{W}^{q+2}}}\m s + \mean{\int_0^t \norm{e^{A_h (t-s)}\,Q_h(R_h -I)B_{N,\delta}(u)}^2_{L^0_2(\mathcal{W}^s)}\m s}\\
  & \leq  N^{-1}\ep^{-2(q+3)-d}Ch^{2\tilde{q}}\int_0^t\mean{\|u\|^2_{\mathcal{W}^{q+3}}}\m s \\
  & \quad + \ep^{-1}h^{2\tilde{q}}\left(\int_0^t \mean{\|B_{N,\delta}(u)\|^2_{\mathcal{W}^{q+2}}}\m s\right) \\
  & \leq N^{-1}\ep^{-2(q+3)-d} C(\delta,T)\left\{1 + \mean{\|u_0\|^2_{\mathcal{W}^{q+3}}}\right\}h^{2\tilde{q}}.
  \end{align*}
  where $\delta$ is the cut-off level of the stochastic integrand $B_{N,\delta}$.
  
   \emph{Term $T_2$}. Using \eqref{e:1000c}, \cref{lemma:rge}, as well as \eqref{BoundWsRIDK}, term $T_2$ is estimated as
     \begin{align*}
  T_2 & = \mean{\norm{\int_0^t e^{A_h (t-s)}\, A_h(\tilde{R}_h-R_h)\,u\,\m s}^2}\\ 
  & \leq C\left\{\mean{\norm{(\tilde{R}_h-R_h)u(t)}^2} + \mean{\norm{\exp(A_ht)(\tilde{R}_h-R_h)u(0)}^2} \right.\\
  & \left.\quad + \mean{\norm{\int_0^t \exp(A_h(t-s))(\tilde{R}_h-R_h)\m u}^2} \right\}\\
  & \leq N^{-1}\ep^{-2(q+3)-d} C(\delta,T)\mean{\|u_0\|^2_{\mathcal{W}^{q+3}}}h^{2\tilde{q}}.
  \end{align*}
  
   \emph{Term $T_3$}. The Lipschitz continuity of $F$, \cref{lemma:rge}, \cite[Proposition 4.36]{Da-Prato2014a} and \eqref{BoundWsRIDK} allow us to deduce
   \begin{align*}
   T_3 & \leq C(T)\int_0^t \mean{\norm{F(u_h(s))-F(u(s))}^2}\,\m s \\
   & \leq C(F,T)\int_0^t \mean{\norm{u_h(s)-u(s)}^{2}}\,\m s \\
   & \leq C(F,T)\left\{\int_0^t \mean{\norm{u(s)-R_h u(s)}^2}\,\m s + \int_0^t \mean{\norm{R_h u(s) - u_h(s)}^2}\m s\right\}\\
  & \leq N^{-1}\ep^{-2(q+2)-d}C(F, T)\mean{\|u_0\|^2_{\mathcal{W}^{q+2}}}h^{2\tilde{q}} \\
  & \quad + C(T,F)\int_0^t \mean{\norm{e_h(s)}^2}\m s.
   \end{align*}
     
  \emph{Term $T_4$}. Since $B_{N,\delta}$ is a Lipschitz approximation of the square root, 
  \cref{lemma:rge} and \eqref{BoundWsRIDK} allow us to deduce   
     \begin{align*}
   T_4 & \leq \int_0^t \mean{\|B_N(u_h(s))-B_N(u(s))\|^2_{L^0_2}}\,\m s \\
   & \leq C(\delta)N^{-1}\ep^{-1}\left\{\int_0^t \mean{\norm{u(s)-R_h u(s)}^2}\,\m s + \int_0^t \mean{\norm{R_h u(s) - u_h(s)}^{2}}\,\m s\right\} \\
  & \leq N^{-1}\ep^{-2(q+2)-d}h^{2\tilde{q}}C(\delta, T) \mean{\|u_0\|^2_{\mathcal{W}^{q+2}}} + CN^{-1}\ep^{-1}\int_0^t \mean{\norm{e_h(s)}^2}\,\m s.
   \end{align*}
   
  \emph{Term $T_5$}. Lemma \ref{LemmaNoiseTailBound} promptly implies that 
     \begin{align}\label{NoiseTruncationError}
  Tr_{L^2(D)}(W_\ep-W_{h,\ep}) \leq \ep^{-d}h^{2\tilde{q}}.
  \end{align} 
  Combining \eqref{NoiseTruncationError} with \eqref{BoundWsRIDK}, and also using the scaling \eqref{ScalingNEpsilon}, we obtain
  \begin{align*}
  T_5 & \leq N^{-1}\ep^{-d}C(d,T)\mean{\|u_0\|^2} h^{2\tilde{q}} \leq C(d,T)\mean{\|u_0\|^2} h^{2\tilde{q}}.
  \end{align*}
Combining all contributions, we obtain
     \begin{align*}
     \mean{\norm{e_h(t)}^2} & \leq N^{-1}\ep^{-2(q+3)-d}C(\delta,T,d)\mean{\|u_0\|^2_{\mathcal{W}^{q+3}}}h^{2\tilde{q}} \\
     & \quad + C_2(F,T)N^{-1}\ep^{-1}\int_0^t \mean{\norm{e_h(s)}^2}\m s,
     \end{align*}     
     and the proof is completed by using the Gronwall lemma and the scaling \eqref{ScalingNEpsilon}.
\end{proof}

\section{Modelling for low-density regime}\label{Positivity}

We propose and discuss  modifications to \eqref{ridk} which address the positivity issue of the density $\rho$. The first modification is applied in all cases: we turn off the noise and potential for $\rho\le 0$ by introducing $\rho^+ \coloneq \max\{\rho,0\}$.
\begin{subequations}
  \makeatletter
  \begin{align*}
    \frac{\partial\rho}{\partial t}
    &=-\nabla \cdot \vec{j}, \\ 
    \frac{\partial \vec{j}}{\partial t}
    &= -\gamma\,\vec{j} - k_BT\,\nabla\rho-\nabla \potential \rho^+ + \sigma\,\frac{1}{\sqrt N}\,\sqrt{\rho^+}\,\vec\xi_\ep.  
  \end{align*}
\end{subequations}
On a modelling basis, these terms do not make sense (as there are no particles). Analytically, the square root $\sqrt{\rho}$ is not well-defined for $\rho<0$ and a regularisation of this type is already part of the well-posedness theory.

\subsection{Extra diffusion}\label{sec:ridk-mix-coeff}

The most obvious way of regularising \eqref{ridk} for positivity is to add extra diffusion to the equation for $\rho$ (i.e., adding the term $D_0\Delta\rho$, $D_0>0$, to the $\rho$-equation) so as to get a strongly damped wave equation. 
Such a system is easy to analyse and maintains the conservation of mass in $\rho$. Additionally, it is easy to simulate and, for large diffusion, it is observed numerically to have positive solutions. See \cref{fig:ridk-mix-coeff} for a one-dimensional example with $D_0=0.5$. This approach leads to very smooth profiles and the stochastic dynamics have largely been lost. There is no obvious way of choosing the diffusion constant $D_0$ and, for example, with a smaller diffusion  $D_0=0.1$ in \cref{fig:diffusion_small}, the density profile becomes negative in some regions of space. More investigations are needed for this correction.

\begin{figure}[h]
\centering
\includegraphics{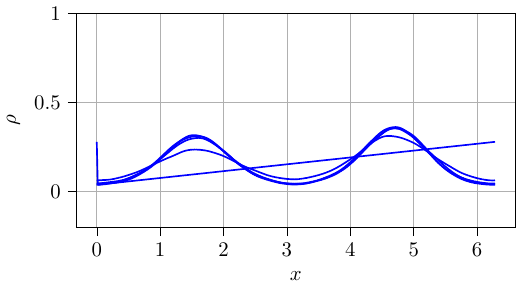}
\caption{Example profiles of $\rho$ from solving \eqref{ridk} with extra diffusion $D_0\Delta\rho$, with $D_0=0.5$. The density profiles of $\rho$ are smooth and always positive in this simulation.}
\label{fig:ridk-mix-coeff}
\end{figure}

\begin{figure}[h]
  \centering
  \includegraphics{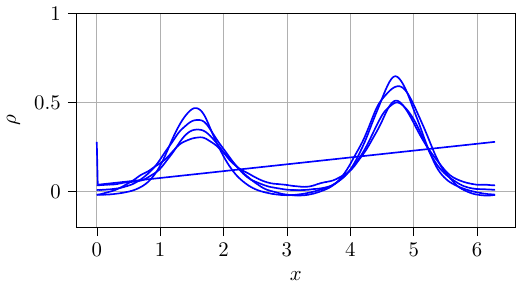}
  \caption{Example profiles of $\rho$ from solving \eqref{ridk} with extra diffusion $D_0\Delta\rho$, with $D_0=0.1$. For large times,  $\rho$ takes some negative values.} 
  \label{fig:diffusion_small}
  \end{figure}

\subsection{Density-dependent time-scales}\label{sec:RIDK-switch}

In this approach, we separate the time scale for the position and momentum in the particle system. In the low-density regime, we speed up the dynamics in the momentum $\vec p$ in a way that causes particles to move more quickly to equilibrium. Intuitively, this causes any excessive momentum density, which may lead to non-negativity, to dissipate. We are able to quantify this analytically and  present a maximum principle-based argument that guarantees non-negativity in \cref{PositivitySwitchHeatRIDK}.  

For the derivation, consider $(\vec p_i,\vec q_i)$ following Langevin dynamics
\[\m \vec q_i=\vec p_i\,\m t,\qquad \m\vec p_i=-\gamma \vec p_i\,\m t +\sigma \,\m \beta_i(t),\]
for i.i.d. Brownian motions $\beta_i(t)$ (we exclude the potential $\potential$ for simplicity). In the derivation of \eqref{ridk}, we now replace $\vec j_\epsilon(\cdot,t)=N^{-1}\sum_i \vec p_i(t) w_\epsilon(\cdot-\vec q_i(t))$ by $\vec J_\epsilon(\cdot,t)=\vec j_\epsilon(\cdot,\tilde{t})$ for $\m\tilde{t}/\m t=1/\varphi_\tau$ for a function $\varphi_\tau>0$ to be specified. Then, we obtain
$$\m\vec J_\epsilon/\m t=N^{-1} \sum_{i=1}^{N} (\m \vec p_i/\m t)\, w_\epsilon(\cdot-\vec q_i) + \vec p_i\, w'_\epsilon(\cdot-\vec q_i) \m \vec q_i/\m t.
$$ 
As $\m \vec p_i(\tilde{t})=-(\m \tilde{t}/\m t)\, \gamma \,\vec p_i(\tilde{t})\,\m t +\sigma \sqrt{\m \tilde{t}/\m t}\,\m \beta_i(t)$, the equation for $\vec J_\epsilon$ becomes
\begin{align}\label{SpeedUpForJ}
  \m \vec J_\epsilon=\left(-\gamma \vec J_\epsilon -k_B \nabla \rho_\epsilon\right)\frac{1}{\varphi_\tau}\,\m t +\sigma\frac{1}{\sqrt{N\varphi_\tau}
}\sqrt{\rho_\epsilon}\,\m \vec \xi_\ep.
\end{align}
When used in the derivation of RIDK, $(\rho_\ep,\vec J_\ep)$ lead to the following variation in the unknown $(\rho,\vec J)\approx (\rho_\ep,\vec J_\ep)$
\begin{subequations}
\begin{align*}
\frac{\partial\rho}{\partial t}
&=-\nabla \cdot \vec J, \\ 
\varphi_\tau(\rho)\frac{\partial \vec J}{\partial t}
&= -\gamma\,\vec J - k_BT\,\nabla\rho -\nabla \potential\,\rho^+ + \sigma\,\frac{1}{\sqrt N}
\,\sqrt{\rho\,\varphi_\tau(\rho)}\,\vec\xi_\ep,
\end{align*}  
\end{subequations}
which is exactly \eqref{Mod}.
Now, dropping the requirement $\varphi_\tau>0$ in \eqref{Mod}, we choose $\varphi_\tau$ to be a smooth monotonic function taking value 0 (respectively, value 1) on the interval $(-\infty, \tau/2)$ (respectively, $(\tau,\infty)$).

We see in \cref{fig:switch-rho} an example of the behaviour of this system with $\tau=0.2$ (all other parameters are as in the simulations for \cref{fig:intro}) and observe non-negative profiles for the density $\rho$.

\begin{figure}[h]
\centering
\includegraphics{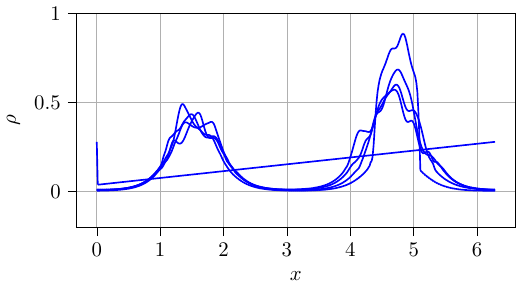}
\caption{Example profiles of $\rho$ from solving \eqref{Mod} with $\tau=0.2$. The density profiles for $\rho$ are always positive in this simulation.}

\label{fig:switch-rho}
\end{figure}

Provided suitable regularity assumption are satisfied (see Remark \eqref{assumptions_positivity}) we can prove that \eqref{Mod} guarantees positivity in the following continuous setting.

\begin{proposition}\label[proposition]{PositivitySwitchHeatRIDK}
Assume that a space- and time-continuous solution $(\rho,\vec J)$ to \eqref{Mod} exists, and that $(\rho,\vec J)$ is twice differentiable in space and once differentiable in time in the region $\{(\f{x},t)\colon \rho(\vec x, t)<0\}$.

Then, provided that $\rho_0>0$, we have $\inf_{\f{x}\in D, t\in[0,T]}\rho(\f{x},t) \geq 0$.
\end{proposition}
\begin{remark}
It is to be noted that the choice of $\varphi_\tau$ in \eqref{Mod} implies that the dynamics of \eqref{ridk} and \eqref{Mod} are identical as long as $\rho \geq \tau$ uniformly.
\end{remark}
\begin{proof}[Proof of Proposition \ref{PositivitySwitchHeatRIDK}]
In order to conclude, we seek to end up in a position where we can apply the standard heat equation maximum principle. Set $Q_t \coloneq D \times (0,t)$ and $w\coloneq\rho+\nu t$.  
Then the system $(w,\vec J)$ solves
\begin{subequations}
\label{w-j-system}
\begin{align}
\frac{\partial w}{\partial t} 
&=-\nabla \cdot \vec J + \nu,\\
\varphi_\tau(w-\nu t)\frac{\partial \vec J}{\partial t}
&= -\gamma\,\vec J - k_BT\,\nabla w -\nabla \potential\,\cdot (w-\nu t)^{+} + \sigma\,\frac{1}{\sqrt N}\,\sqrt{\rho\,\varphi_\tau(w - \nu t)}\,\vec\xi_\ep.
\end{align}
\end{subequations}
We distinguish two cases.

\emph{Case 1}. It holds $\min_{Q_{T-\nu}}w \geq 0$: then $\min_{Q_{T-\nu}}\rho \geq -\nu T$. 

\emph{Case 2}. It holds $\min_{Q_{T-\nu}}w < 0$. The definition of $\varphi_\tau$ and the fact that $w-\nu t \leq w $ imply that, at $(\f{x}_{\min,\nu}, t_{\min,\nu})\coloneq \arg\min_{Q_{T-\nu}}{w}$, the system \eqref{w-j-system} reduces to  
\begin{subequations}
\label{w-j-system-reduced}
\begin{align*}
\frac{\partial w}{\partial t} 
&=-\nabla \cdot \vec J + \nu,\\
0 &= -\gamma\, \vec J - k_BT\,\nabla w ,
\end{align*}
\end{subequations}
or, equivalently, using the regularity of $(\rho, \vec j)$ at $(\f{x}_{\min,\nu}, t_{\min,\nu})$,
\begin{align*}
\frac{\partial w}{\partial t} = \frac{k_B T}{\gamma}\Delta w + \nu.
\end{align*}
At this stage, one can apply the standard contradiction for the heat equation maximum principle (granted by the fact that $\partial w/\partial t \leq 0$ and $\Delta w \geq 0$), and deduce that $\min_{Q_{T-\nu}}{w} = \min \rho_0$. 

Putting the two cases together, we obtain $\min_{Q_{T-\nu}}{w} \geq \min\{-\nu T; \min \rho_0\}$. Using the continuity of $\rho$ and the definition of $w$, we conclude by writing
\begin{align*}
\min_{Q_{T}}{\rho} & = \lim_{\nu\rightarrow 0}{\min_{Q_{T-\nu}}{\rho}} \geq  \lim_{\nu\rightarrow 0}{\left\{\min_{Q_{T-\nu}}{w} -\nu T\right\}} \geq \min\{0;\min\rho_0\}.    \qedhere
\end{align*}
\end{proof}

\begin{remark}\label{assumptions_positivity}
Even though Proposition \ref{PositivitySwitchHeatRIDK} is a step in the right direction when it comes to models which preserve positivity of the density, its application relies on a well-posedness theory for the solution (in particular, suitable space and time differentiability in the region $\{(\f{x},t)\colon \rho(\vec x, t)<0\}$): such a well-posedness theory is still missing, and is deferred to future works. 
\end{remark}

\begin{remark}
As far as maximum principles in the discrete setting are concerned, we have so far looked at schemes at the prototype scheme
\begin{align}
\frac{\rho_h(\f{x},t)-\rho_h(\f{x},t-\delta t)}{\delta t} & = -\nabla_h \cdot \vec{j}_h(\f{x},t),  \label{who} \\
\varphi_\tau\frac{\vec{j}(\f{x},t)-\vec{j}(\f{x},t-\delta t)}{\delta t}
&= -\gamma\,\vec{j}_h(\f{x},t) - k_BT\,\nabla_h\rho_h(\f{x},t) \nonumber\\
&\quad + \sigma\,\frac{1}{\sqrt N}\,\sqrt{\rho_h(\f{x},t-\delta t) \varphi_\tau} (\vec\xi_{\ep,h}(\f{x},t) - \vec\xi_{\ep,h}(\f{x},t-\delta t)),\label{boo}
\end{align}
where $\delta t$ is a timestep, $\nabla_h \cdot $ and $\nabla_h$ are (non-local) numerical discretisations of the divergence and gradient, such that the operator $ \nabla_h \cdot  \nabla_h$ is non-negative. 

The aim is to choose $\varphi_\tau$ so as to obtain a discrete maximum principle.
Mimicking \cref{PositivitySwitchHeatRIDK}, suppose that $w_h=\rho_h+\nu\,t$ for some $\nu>0$ attains its minimum at $x_{\min},t_{\min}$ with a negative value, $w_h(x_{\min}, t_{\min})<0$. To eliminate the left-hand side of \eqref{boo}, we need to impose $\varphi_\tau=0$ at $(x_{\min}, t_{\min})$. 
It must also be zero at neighbouring values of $x_{\min}$ in order to evaluate the term $\nabla_h \cdot \vec j_h$ and at neighbouring values of $t_{\min}$ to eliminate the noise term. 
With this assumption, we find $0=-\gamma \vec j_h - k_B T\, \nabla_h \rho_h$ and can now deduce that 
\begin{align*}
\frac{w_h(x_{\min},t_{\min}) - w_h(x_{\min},t_{\min} - \delta t)}{\delta t} = \frac{k_B T}{\gamma}\nabla_h \cdot \nabla_h w_h(x_{\min}, t_{\min}) + \nu.
\end{align*}
The left-hand side is negative and, if $\nabla_h\cdot \nabla_h$ is a non-negative operator, the right-hand side is positive, leading to a contradiction.

The assumptions so far discussed for $\varphi_\tau$ are quite demanding, and it is unclear how to make practical choices for $\varphi_\tau$.
So far, we have found no scheme for which we can guarantee positivity as a computationally convenient $\varphi_\tau$ is not forthcoming. The semi-implicit time-stepper that has been implemented with DG (which uses the local, more simplistic choice of $\varphi_\tau$ given for the continuous case, see discussions following \eqref{Mod}) can lead to negativity of the density; see for example \cref{fig:tau_small}. In practise, we are able to avoid negative regions by increasing $\tau$ and refining the discretisation parameters. Again, more analysis is required to quantify the preceding considerations.
\end{remark}

\begin{figure}[h]
  \centering
  \includegraphics{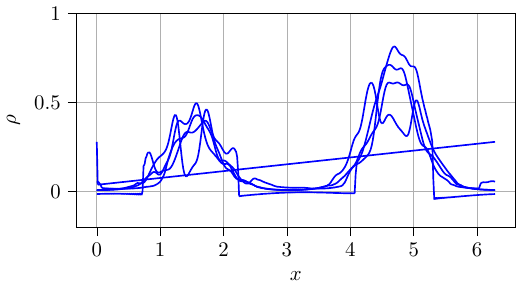}
  \caption{Example profiles of $\rho$ from solving \eqref{Mod} with $\tau=0.05$. The density profile $\rho$  takes some negative values near $t=10$.}
  \label{fig:tau_small}
  \end{figure}

\section{example: ridk for reacting/diffusing populations in two dimensions}\label{extended}
  
In order to demonstrate the applicability of our DG framework, we focus on a DK model associated with a system of reacting/diffusing particles in two dimensions: first, we describe the underlying particle model and then present a DK model for the same dynamics. We compare numerical simulations of such a model to the particle system (using the time-scale regularisation presented in subsection \ref{sec:RIDK-switch}) and verify that, under suitable conditions, the microscopic dynamics can be replicated. 

\subsection{The particle system}
Consider $N$ particles with position $q_i$ and momentum $p_i$ of type $T_i\in\{A,B\}$ following Langevin dynamics
\begin{align*}
  \m q_i=p_i\,\m t,\qquad \m p_i=-\gamma \,p_i \,\m t - \nabla \mathcal{V}(q_i)\,\m t + \sigma\,\m \beta_i(t),
\end{align*}
for dissipation $\gamma$ and noise coefficient $\sigma$ in an external potential $\mathcal V$. The particles react as $A+B \mapsto  2B$ with rate $\kappa$ in a ball of radius of size $r$ (that is, if one particle each of type $A$ and $B$ are within distance $r$ of each other, the type-$A$ particle changes type with probability $1-\exp(-\kappa\,\Delta t)$ on a time interval of length $\Delta t$). 

In the numerical experiments that follow, $\gamma=0.3$, $\sigma=0.2$, $\mathcal{V}(x,y)=\frac18 (\cos(y/2)^2+2\cos(1+x/2)^2)$. There are initially $N=5000$ particles, consisting of $N_A=4500$ particles of type $A$ and $N_B=500$ particles of type $B$ reacting in a ball of radius $r=0.15$ with rate $\kappa=0.2$.  Initially, the particles have zero momenta and positions given by i.i.d.  samples from the normal distributions $\Nrm(\mu_i, \sigma^2_i I)$ for $\mu_A=[4.5,1.5]$, $\mu_B=[4.2,5]$, $\sigma_A=0.8$, $\sigma_B=0.25$. Snapshots of the densities found by simulating the particle model are shown in \cref{buntag1,buntag2,buntag3}. Particles of different types start off separated according to the initial distributions (\cref{buntag1}),  fall into the potential well and start mixing (\cref{buntag2}), before rapidly converting to type-$B$ particles (\cref{buntag3}). This example follows \cite{helfmann2021interacting}, which has a similar example in one dimension for the over-damped case.

\begin{figure}[h]
  \centering
  \includegraphics[scale=1]{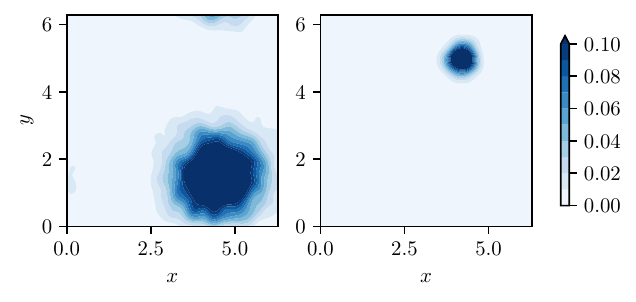}
  \caption{Particle model at $t=0$: left-hand (respectively, right-hand) plot shows the initial density of type-$A$ (respectively, type-$B$) particles.}\label{buntag1}
  \end{figure}

  \begin{figure}[h]
    \centering
    \includegraphics[scale=1]{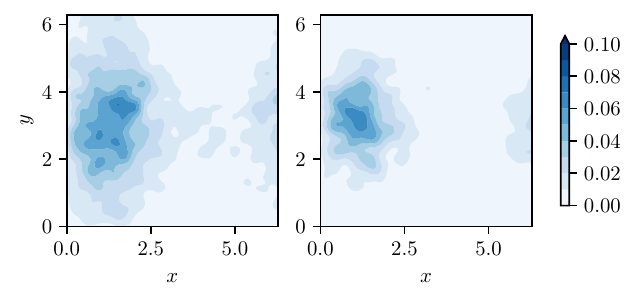}
    \caption{Densities for the particle model at $t=12$. Both types of particle move into the well of the potential $\mathcal{V}(x,y)=\frac18 (\cos(y/2)^2+2\cos(1+x/2)^2)$ with minimum $(x,y)=(\pi-2,\pi)$.}\label{buntag2}
    \end{figure}

\begin{figure}[h]
  \centering
  \includegraphics[scale=1]{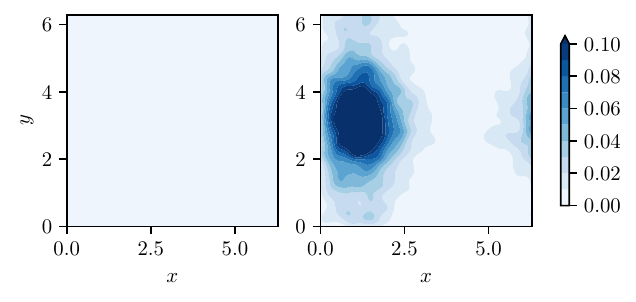}
  \caption{Densities for the particle model at $t=25$. The reaction has converted most particles to type-$B$, which inhabit the bottom of the potential well.}\label{buntag3}
  \end{figure}
  
\subsection{The associated RIDK dynamics}  

The Dean--Kawasaki version of this model is the following coupled system of SPDEs:
\begin{align*}
  \frac{\partial\rho^A}{\partial t}
  &=-\nabla \cdot \vec{j}^A - k\,\pi r^2\,N\,\rho^A\, \rho^B\,1_{\rho_B>\rho_{\mathrm{th}}},\\
  \varphi_\tau(\rho^A)\frac{\partial \vec{j}^A}{\partial t}
  &= -\gamma\,\vec{j}^A - k_B T\,\nabla\rho^A -\nabla \potential \,\rho^A+ \sigma\,\frac{1}{\sqrt N}\,\sqrt{\rho^A\,\varphi_\tau(\rho^A)}\,\vec\xi^A_\ep,\\
  \frac{\partial\rho^B}{\partial t}
  &=-\nabla \cdot \vec{j}^B + k\,\pi r^2\,N\,\rho^A\, \rho^B\,1_{\rho_B>\rho_{\mathrm{th}}},\\
 \varphi_\tau(\rho^B) \frac{\partial \vec{j}^B}{\partial t}
  &= -\gamma\,\vec{j}^B - k_B T\,\nabla\rho^B -\nabla \potential \,\rho^B+ \sigma\,\frac{1}{\sqrt N}\,\sqrt{\rho^B \,\varphi_\tau(\rho^B)}\,\vec\xi^B_\ep.
\end{align*}
This is derived by taking two separate particle and momentum densities $\rho^i, \vec j^i$ for $i\in\{A,B\}$ and two independent copies $\vec \xi_\epsilon^i$ of the RIDK noise term $\vec \xi_\epsilon$. Given the interaction radius is $r$ and there are $N$ particles, in two dimensions, the RIDK equations for $\rho^i, \vec j^i$ are coupled by the reaction term $\kappa\,\pi r^2\,N\,\rho^A\,\rho^B\,1_{\rho_B>\rho_{\mathrm{th}}}$, where $\kappa$ is the reaction rate and $\rho_{\mathrm{th}}$ is a threshold for $B$ particles before reaction is allowed. 
Without this factor, the exponential growth starts very early as the particle density is much more widely spread than for the particle model.

\subsection{Simulations}
To match the particle simulation, we take $\vec j^A=\vec j^B=0$ at time $t=0$, and $\rho^i$ as the pdf of the density $\Nrm(\mu_i, \sigma_i^2)$ scaled by the type mass $N^i/N$ (of particles of type $i\in\{A,B\}$). We take $\rho_{\mathrm{th}}=0.012$ to match the behaviour of the particle system.

  \begin{figure}[h]
    \centering
    \includegraphics[scale=1]{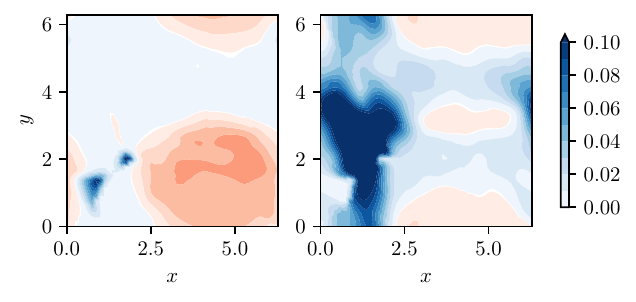}
    \caption{Particle densities $\rho^i$ (left $i=A$; right $i=B$) for the RIDK model at $t=10$. The orange regions indicate negative densities ($\rho^i<0$).}\label{gabii2}
    \end{figure}

\begin{figure}[h]
  \centering
  \includegraphics[scale=1]{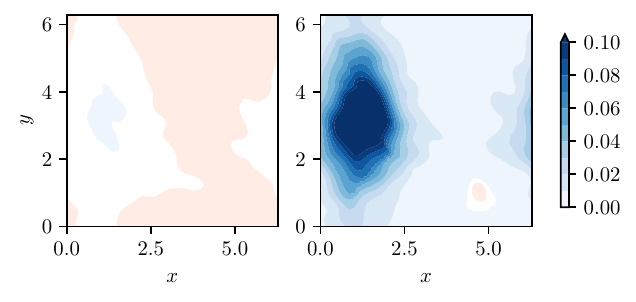}
  \caption{Particle densities for the RIDK model at $t=25$, where most of the mass is shown on the right (type-$B$ particles) and some negative density remains on the left in the type-$A$ density.}\label{gabii3}
  \end{figure}

  See \cref{gabii2,gabii3} for simulation of the unregularised system ($\varphi_\tau(\rho)\equiv 1$). We see similar dynamics to the particle model with the density first concentrating in the well, where the particles react and rapidly convert to type $B$. There are  regions of negative particle densities (indicated by orange) as we saw in one-dimension (\cref{fig:intro}).

\begin{figure}[h]
  \centering
  \includegraphics[scale=1]{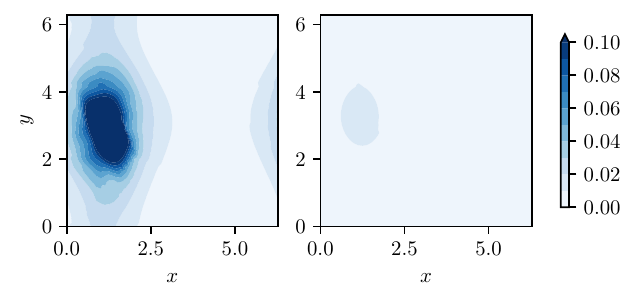}
  \caption{Particle densities for the RIDK model (time-scale regularisation with $\tau=0.05$) at $t=10$.}\label{hapon2}
  \end{figure}

\begin{figure}[h]
\centering
\includegraphics[scale=1]{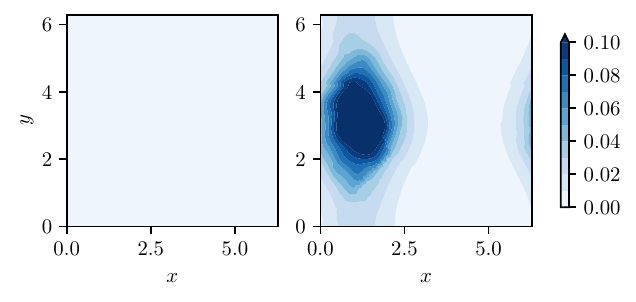}
\caption{Particle densities for the RIDK model (time-scale regularisation) at $t=25$.}\label{hapon3}
\end{figure}

\clearpage
Next we simulate RIDK with the time-scale regularisation ($\varphi_\tau$ defined by \cref{sec:RIDK-switch} with $\tau=0.05$). The simulations are shown in \cref{hapon2,hapon3}. There are no longer any regions of negative density. We compare the evolution of total probability mass of particle~$B$ in \cref{massmass}.  The DG simulation  conserves total mass $\int_D \rho^A(t,x)+\rho^B(t,\vec x)\m\vec x=1$. However, we see the mass of the $B$ particles overshoot one in \cref{massmass}, as there are regions with negative density for the $A$ particles (see \cref{hapon2}). When the time-scale regularisation is used, the transition $A\mapsto B$ is more sudden and the mass of $B$ particles never overshoots.
\begin{figure}[h]
    \centering
    \includegraphics{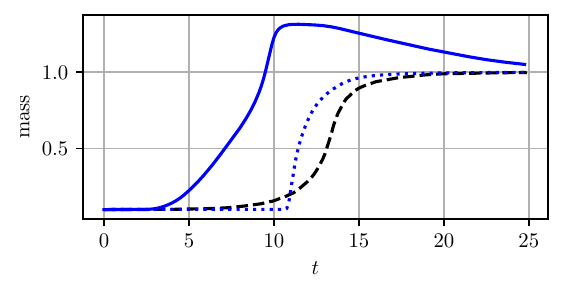}
    \caption{Type-$B$ mass profiles for particle model (black dashed) and RIDK system unregularised (blue solid) and regularised (dotted blue; $\tau=0.02$). Notice that the blue solid line overshoots and exceeds total-mass one (it is compensated by negative mass in type $A$).}\label{massmass}
    \end{figure}

\subsection{Conclusions} Using a RIDK model to simulate a system of diffusing/interacting particles appears to be effective and physically plausible: in particular, specific truncations on the densities grant non-negativity of the densities, and a good description for the transfer of mass (from type $A$ to type $B$). 

Many aspects remain open. Firstly, the simulations are still quite sensitive to the specific regularisation and truncation levels chosen for the density, and more insight is needed to address this point. Secondly, for the sake of simplicity, our RIDK model does not include noise fluctuations at the level of the particle reaction: it would be interesting to assess the impact of adding such a noise to the model (i.e., in the spirit of \cite{kim2017stochastic}).

\appendix
 
\section{selected technical features of ridk model}\label{AppRIDKSelectedFeatures} 
    
\subsection{Noise expansion}\label{AppNoiseExpansion}    

Let $D\coloneq \mathbb{T}^d$. For each $\ep>0$, we define the \emph{von Mises} kernel as
\begin{align}\label{Covariance}
w_\ep(\f{x}-\f{y}) \coloneq Z^{-d}_{\ep}\exp\left\{-\frac{\sum_{\ell=1}^{d}{\sin^2((x_\ell-y_\ell)/2)}}{\ep^2/2}\right\},\qquad \f{x},\f{y}\in D,
\end{align}
for some normalisation constant $Z_\epsilon$.
The noise $\vec\xi_\ep$ in \eqref{ridk} can be explicitly expanded using the spectral properties of the convolution operator $P_{\ep}\colon L^2(D)\rightarrow L^2(D)\colon f\mapsto w_\ep \ast f$, which, due
to the separability of the kernel $w_{\ep}$, are readily available from the one-dimensional case~\cite[Section
  4.2]{Cornalba2019a}. More specifically, with $\{e_j\}_{j\in\mathbb{Z}}$ being the trigonometric system
\begin{align*}
 e_{j}(x) \coloneqq 
  \begin{cases}
    \pi^{-1/2}\cos(jx), & \mbox{if } j>0, \vspace{0.3 pc}\\
    \pi^{-1/2}\sin(jx), & \mbox{if } j<0, \vspace{0.3 pc}\\
    (2\pi)^{-1/2}, & \mbox{if } j=0,
  \end{cases}                   
\end{align*}
it is not difficult to see that the family $\{f_{\f{j},s}\}_{\f{j}\in\mathbb{Z}^d}$ defined as
\begin{align*}
  f_{\f{j},s}(\f{x}) \coloneqq  
    C(d)\left\{\prod_{\ell=1}^{d}{e_{j_{\ell}}(x_{\ell})}\right\}\left(1+|\f{j}|^2\right)^{-s/2}, \quad \f{j}\in\mathbb{Z}^d,
\end{align*}
is, for some suitable normalisation constant $C(d)$, an $H^s$-orthonormal basis of eigenfunctions for $P_{\sqrt{2}\ep}$ for any
$\ep>0$. Furthermore, the eigenvalue of $P_{\sqrt{2}\ep}$ corresponding to the eigenfunction $f_{\f{j},s}$ is
\begin{align}\label{NoiseEigenvalues}
  \lambda_{\f{j},\ep} & = \prod_{\ell=1}^{d}{\lambda_{j_{\ell},\ep}},
\end{align}
where the eigenvalues from the one-dimensional case are given by
\begin{align*}
  \lambda_{j,\ep} 
  & =
       \mathbf{1}_{j=0} + \mathbf{1}_{j\neq 0}\displaystyle Z^{-1}_{\sqrt{2}\ep}\int_{\mathbb{T}}{e^{-\frac{\sin^2(x/2)}{\ep^2}}\cos(jx)
      \m x}=I_j\left(\{2\ep^2\}^{-1}\right) / I_0\left(\{2\ep^2\}^{-1}\right), 
      \end{align*}
with $I_j$ denoting the modified Bessel function of first kind and order $j$~\cite[Eq.~(9.6.26)]{abramowitz1964handbook}. As a result,
the stochastic process
\begin{align}
  \label{eq:unreg-noise}
  W_{\ep} & \coloneqq \sum_{\f{j}\in\mathbb{Z}^d}{\sqrt{\alpha_{\f{j},s,\ep}}\,(0,f_{\f{j},s},0,\dots,0)}\,\beta_{1,\f{j}}+\cdots\nonumber\\
  & \quad + \sum_{\f{j}\in\mathbb{Z}^d}
  {\sqrt{\alpha_{\f{j},s,\ep}}\,(0,\dots,0,f_{\f{j},s})}\,\beta_{d,\f{j}},\qquad \alpha_{\f{j},s,\ep}
  \coloneqq (1+|\f{j}|^2)^{s}\lambda_{\f{j},\ep},
\end{align}
with i.i.d. families $\{\beta_{\ell,\f{j}}\}_{\ell=1}^{d}$ of independent Brownian motions, is a $\mathcal{W}^s$-valued $Q$-Wiener
process representation of the $\mathbb{R}\times\mathbb{R}^d$-valued stochastic noise $(0,\vec\xi_\ep)$, where $\mathcal{W}^s$ is defined in \eqref{Ws}.

We also prove a handy result concerning the decay of the sequence $\{\lambda_{\f{j},\ep}\}_{\f{j}\in \mathbb{Z}^d}$ as defined in \eqref{NoiseEigenvalues}, which is directly related to \cite[Lemma 3.2]{cornalba2021well}. 

\begin{lemma}\label{LemmaNoiseTailBound}
For $h$ small enough, the following bound holds
\begin{align*}
\sum_{\f{j}\in \mathbb{Z}^d\colon |\f{j}|_1>\ep^{-1}|\ln(h^{2\tilde{q}})|}{ \lambda_{\f{j},\ep}} \leq C(d)\ep^{-d}h^{2\tilde{q}}.
\end{align*}
\end{lemma}
\begin{proof}
Consider the case $d=1$ first. We take $\alpha = \beta = 1/2$ in \cite[Lemma 3.2]{cornalba2021well}. This means that 
\begin{align}\label{EigenvaluesExponetialDecay}
\lambda_{j,\ep} \leq \left(\frac{\ep^{-1}/\sqrt{2}-1}{\ep^{-1}/\sqrt{2}}\right)^{j - \ep^{-1}/\sqrt{2}}\quad \mbox{for }j > \ep^{-1}/\sqrt{2}.
\end{align}
Additionally, for small enough $h$, we obtain
\begin{align}\label{NumberStepsBeforeExponentialDecay}
\ep^{-1}|\ln(h^{2\tilde{q}})| - \ep^{-1}/\sqrt{2} \geq (\ep^{-1}/\sqrt{2})|\ln(h^{2\tilde{q}})|.
\end{align}
Combining \eqref{EigenvaluesExponetialDecay} and \eqref{NumberStepsBeforeExponentialDecay} gives
\begin{align}
\sum_{j\in \mathbb{Z}\colon |{j}|>\ep^{-1}|\ln(h^{2\tilde{q}})|}{ \lambda_{j,\ep}} & \leq 2\sum_{n\geq 0}{\left(\frac{\ep^{-1}/\sqrt{2}-1}{\ep^{-1}/\sqrt{2}}\right)^{n + (\ep^{-1}/\sqrt{2})|\ln(h^{2\tilde{q}})|}} \nonumber\\
& \leq 2(\ep^{-1}/\sqrt{2}) (1/e)^{-\ln(h^{2\tilde{q}})} \leq 2\ep^{-1}h^{2\tilde{q}}.
\end{align}
The extension to arbitrary $d>1$ is done by observing that 
\begin{align*}
\sum_{|\f{j}|_1>\ep^{-2}|\ln(h^{2\tilde{q}})|}{ \lambda_{\f{j},\ep}} \leq \ep^{1-d}\sum_{\ell = 1}^{d}\sum_{|j_{\ell}|>2\ep^{-1}|\ln(h^{2\tilde{q}})|}{ \lambda_{\f{j},\ep}}.
\end{align*}
\end{proof}

\subsection{Proof of Proposition \ref{prop-pathwise}}\label{app:pathwise}
  
The validity of \eqref{MildSolution} and \eqref{BoundWsRIDK} is settled using \cite[Theorem 1.1]{cornalba2021well}. We now proceed to the proof of \eqref{PathwiseSmoothedIntegrand}.
We exploit the equivalence between different notions of solutions to SPDEs, as presented in \cite[Appendix F]{Prevot2007a}. We split the proof in several steps.\\

\emph{Step 1: Basic regularity of $u$}. As $\overline{s}-1>d/2$, $u$ solves 
\begin{align}\label{ridk-reg}
  u(t) & = S(t)\,u_0 
   +\int_{0}^{t}{S(t-s)\,F(u(s))\,\m s} 
    +\int_{0}^{t}{S(t-s)B_{N,\delta}(u(s))\,\m W_{\ep}(s)}
 \end{align}
 on the probability space $\Omega$ and up to some time $T$, 
 where $B_{N,\delta}$ is Lipschitz with respect to the $\mathcal{W}^{\overline{s}-1}$ norm. Using the a priori estimates as in \cite[Theorem 1.1]{Cornalba2021al}, we get
\begin{align}\label{d:10}
 \mean{\int_{0}^{T}{\|B_{N,\delta}(u(t))\|^2_{L^0_2(\mathcal{W}^{\overline{s}-1})}\,\m t}} < \infty.
\end{align}

\emph{Step 2: $u$ is a mild solution $\Longrightarrow$ $u$ is an analytically weak solution}. In this step, we want to apply \cite[Appendix F, Proposition F.0.5(ii)]{Prevot2007a}.
Inequality~\eqref{d:10} allows us to use \cite[Proposition 6.2]{Da-Prato2014a} to deduce that the stochastic integral in \eqref{ridk-reg} has a predictable version. Additionally, for any $\zeta \in A^{\ast}$, with the adjoint operator $A^\ast$ being
$$
A^\ast\colon H^{\overline{s}}\times \f{V}^{\overline{s}}\colon (\phi,\f{\psi}) \mapsto (\nabla \cdot \f{\psi}, -\gamma \f{\psi} + \nabla \phi),
$$
where $\f{V}^{s}$ is defined in \eqref{VsPlus1}, 
we have
\begin{align*}
& \int_{0}^{T}{\mean{\int_{0}^{t}{\|\langle S(t-\tilde{t})B_{N,\delta}(u(\tilde{t})), A^\ast \xi\rangle\|_{L^0_2(\mathcal{W}^{\overline{s}-1})}\, \m \tilde{t}}}\m t} \\
& \quad \leq \int_{0}^{T}{\mean{\int_{0}^{t}{\|\langle B_{N,\delta}(u(\tilde{t}))\|_{L^0_2(\mathcal{W}^{\overline{s}-1})} \|A^\ast \xi\|_{\mathcal{W}^{\overline{s}-1}}\,\m \tilde{t}}}\m t} \stackrel{\eqref{d:10}}{<} \infty.
\end{align*}
Furthermore, it is immediate to see that $\mathbb{P}(\int_{0}^{T}{\|F(u(t))\|_{\mathcal{W}^{\overline{s}-1}}\m t} < \infty) = 1$.
Therefore, we have verified all assumptions of \cite[Appendix F, Proposition F.0.5(ii)]{Prevot2007a}, and we use it to deduce that $u$ is an analytically weak solution.

\emph{Step 3: $u$ is an analytically weak solution $\Longrightarrow$ $u$ is an analytically strong solution}. In this step, we want to apply \cite[Appendix F, Proposition F.0.4(ii)]{Prevot2007a}.
The process $u$ takes values in $\mathcal{W}^{\overline{s}} \subset \mathcal{D}(A) = H^{\overline{s}} \times \f{V}^{\overline{s}}$ due to the assumption $u_0 \in \mathcal{W}^{\overline{s}}$ (and the same existence theory described above with $\overline{s}$ replacing $\overline{s}-1$). Furthermore, using once again the a priori estimates as in \cite[Theorem 1.1]{Cornalba2021al}, we deduce that 
\begin{align*}
& \mathbb{P}\left(\int_{0}^{T}{\|A u(t)\|_{\mathcal{W}^{\overline{s}-1}}\m t} < \infty\right)=1, \qquad \mathbb{P}\left(\int_{0}^{T}{\|F(u(t))\|_{\mathcal{W}^{\overline{s}-1}}\m t} < \infty\right) = 1,\\
& \mathbb{P}\left(\int_{0}^{T}{\|B_{N,\delta}(u(t))\|^2_{L^0_2(\mathcal{W}^{\overline{s}-1})}\,\m t <\infty}\right) = 1.
\end{align*}

We have verified all assumptions of \cite[Appendix F, Proposition F.0.4(ii)]{Prevot2007a}, and we use it to deduce that $u$ is an analytically strong solution. Therefore, \eqref{PathwiseSmoothedIntegrand} is proved. 

\emph{Step 4: $u$ is a high-probability path-wise solution to \eqref{ridk}}: this follows from Remark \ref{rem:1}.

\section{Relevant It\^o differentials}\label{RelevantItoDifferentials}
 
\begin{lemma}[Relevant vector-valued It\^o differentials]\label{ItoDifferential}
Let $\overline{s}$ satisfy \eqref{ScalingNEpsilon}, and let $u$ be the process solving \eqref{PathwiseSmoothedIntegrand}. For any fixed $t>0$, consider the functionals
\begin{align*}
g_1(z) & \colon \mathcal{W}^{\overline{s}-1} \mapsto V_h\colon z\mapsto (\tilde{R}_h-R_h)z,\\
g_2(z) & \colon \mathcal{W}^{\overline{s}-1} \mapsto V_h\colon z\mapsto Q_h(R_h - I)z,\\
g_3(s,z) & \colon \mathcal{W}^{\overline{s}-1} \mapsto V_h\colon z\mapsto e^{(t-s)A_h}(\tilde{R}_h-R_h)z.
\end{align*}
where $R_h,\tilde{R}_h$ have been introduced in Section \ref{ErrorAnalysisLinearPDE}.
Then we have
\begin{align}
\emph{\m} g_1(u) & = (\tilde{R}_h-R_h)\emph{\m} u, \label{e:1000a}\\
\emph{\m} g_2(u) & = Q_h(R_h-I)\emph{\m} u, \label{e:1000b}\\
\emph{\m} g_3(s,u) & = -e^{(t-s)A_h}A_h(\tilde{R}_h-R_h)u\emph{\m} s + e^{(t-s)A_h}(\tilde{R}_h-R_h)\emph{\m} u.\label{e:1000c}
\end{align}
\end{lemma}
\begin{proof}
We use the vector-valued It\^o formula given in \cite[Theorem 3.8]{curtain1970ito}, on the time interval $[0,t]$, and where $K \coloneq \mathcal{W}^{\overline{s}-1}$ and $G \coloneq V_h$.

As for $g_1$, assumptions (i), (iii) and (vi) of \cite[Theorem 3.8]{curtain1970ito} are trivially satisfied. Assumption (ii) holds as $\tilde{R}_h-R_h \in \mathcal{L}(\mathcal{W}^{\overline{s}-1},V_h)$: this can be easily deduced using \cref{nearly_coercive}, \cref{bdd_ah} and standard interpolation estimates from \cite[Chapter 3]{quarteroni2008numerical}.
Assumption (iv) holds since $Au \in \mathcal{W}^{\overline{s}-1}$, and (v) holds since the stochastic integrand $B_N(u)$ is Lipschitz, so all moments (in particular, the fourth moment) can be bounded. We can use \cite[Theorem 3.8]{curtain1970ito} and get \eqref{e:1000a}. 
Similar considerations settle also \eqref{e:1000b}.
As for $g_3$, everything is identical except for point (i) (the time differentiability of $g_3$ is trivial given the fact that the exponential $e^{(t-s)A_h}$ has a finite-dimensional input). Using \cite[Theorem 3.8]{curtain1970ito} and \eqref{e:1000a} we get \eqref{e:1000c}.
\end{proof}

{\bfseries Acknowledgements}. 
The authors thank the anonymous referees for their careful reading of the manuscript and their valuable suggestions.
FC gratefully acknowledges funding from the Austrian Science Fund (FWF) through the project F65, and from the European Union’s Horizon 2020 research and innovation programme under the Marie Sk\l odowska-Curie grant agreement No. 754411 (the latter funding source covered the first part of this project).

\printbibliography

\end{document}